\documentclass[11pt]{article}
\usepackage{amsmath, amsthm, amsfonts, amssymb, color}
\usepackage{mathrsfs}

\newtheorem{thm}{Theorem}[section]

\newtheorem{lemma}[thm]{Lemma}
\newtheorem{proposition}[thm]{Proposition}

\theoremstyle{definition}

\newtheorem{remark}{Remark}[section]
\newtheorem{asu}{Assumption}[section]
\newcommand{\scr}[1]{\mathscr #1}
\definecolor{wco}{rgb}{0.5,0.2,0.3}

\numberwithin{equation}{section}

\def\1{{\mathbf 1}}
\def\R{{\mathbf R}}
\def\F{{\mathcal F}}
\def\I{{\mathcal I}}

 \newcommand{\Span}{{\mathrm span}}
\def\E{{{\bf E}}}
\def\p{{{\bf P}}}
\def\loc{\text{\rm{loc}}}
\def \lloc{{L_{loc}}}

 \def\D{\scr D}
\def\e{\text{\rm{e}}}   
\def\paral{/\kern-0.55ex/} 
\def\parals_#1{/\kern-0.55ex/_{\!#1}} 
\def\Ric{\mathop{\rm Ric}}
\def\le{{\leqslant}}
\def\ge{{\geqslant}}
\def\<{{\langle}}
\def\>{{\rangle}}
\def\epsilon{{\varepsilon}}

\begin{document}

\title{An approximation scheme for  SDEs with non-smooth coefficients}
\author{Xin Chen and Xue-Mei Li\footnote{
Research supported by  EPSRC grant (EP/E058124/1).
 }}

\date{}
\maketitle

\begin{abstract}
Elliptic stochastic differential equations (SDE) make sense when the coefficients are 
only continuous. We study  the corresponding linearized SDE 
whose coefficients are not assumed to be locally bounded.  This leads to existence of $W_{\loc}^{1,p}$ solution flows for elliptic SDEs with H\"older continuous and $\cap_{p} W_{\loc}^{1,p}$ coefficients. 
Furthermore an approximation scheme is studied from which we obtain a representation for the derivative of the Markov semigroup, %Malliavin differentiability of the solution 
and an integration by parts formula. 
%Limits of intrinsic covariant differentiation of the system is investigated in terms of the splitting of the noise into an intrinsic and a redundant part. An intrinsic integration by part formula for the measure induced by the elliptic SDE on the path space are obtained.
\end{abstract}

\section{Introduction}
Let $A_l, 1 \le l \le m$, be   continuous vector
fields on $\R^n$.
We consider  stochastic differential equations of Markovian type
\begin{equation} \label{e1}
d\xi_t=\sum_{l=1}^{m}A_l(\xi_t)dW_t^l+A_0(\xi_t)dt
\end{equation}
where $(W_t^l, 1\le l\le m )$ are independent Brownian motions. Denote by $A_{il}$ the components of $A_l$ hence $A_l=(A_{1l}, A_{2l}, \dots,A_{nl})^T$. Write  $A=(A_1,\dots, A_m)$ for the $n\times m$ matrix with the induced $n\times n$-matrix $A^{\ast}A$ whose entries are
 $a_{ij}(x)=\sum_{l=1}^m A_{il}(x)A_{jl}(x)$.
For $x\in \R^n$ let $\xi_t(x)$ be a solution to the SDE (\ref{e1}) with initial value $x$. If the function $A_l$ are weakly differentiable there is formally the linearized SDE, 
\begin{equation}\label{e4}
V_t(x)= \mathbf{I}+\sum_{l=1}^{m}\int_0^t DA_l(\xi_s(x))(V_s(x))dW_s^l+ \int_0^tDA_0(\xi_s(x))(V_s(x))ds
\end{equation}
whose solution is a $n\times n $ matrix valued random function. Here $DA_l:\R^n\rightarrow L(\R^n, \R^n)$ denotes the function $DA_l(x)=(DA_{1l}(x),...,DA_{nl}(x))$.
 In the case of  the vector fields $A_l$ being smooth and  when  a global smooth solution flow to SDE (\ref{e1}) exists, the solution to (\ref{e4}) corresponds to the derivative of the solution (\ref{e1}) with respect to initial data.  Here in Section 2 of the paper we do not assume local boundedness of $DA_l$. We state below our two basic sets of assumptions, which are used to show a key convergence theorem. 
%For the Bismut type formula and for the integration by parts formula these assumptions can be relaxed. 
\begin{asu}\label{N1}
\begin{enumerate}
\item [(1)] $A(x)$ is uniformly elliptic, for some $\theta>0$,
\begin{equation}\label{c1}
\sum_{i,j=1}^n a_{ij}(x)\xi_i \xi_j\geqslant \theta |\xi|^2, 
\quad  \forall  x \in \R^n,\ \xi=(\xi_1,...,\xi_n)\in \R^n.
\end{equation}
\item[(2)]  Each $A_l(x), 0\leqslant l \leqslant m$ is uniformly H\"older continuous , i.e. for some positive
$K$ and $0<\alpha<1$, 
\begin{equation*}
|A_l(x)-A_l(y)|\leqslant K|x-y|^{\alpha}\  \ \ \forall  x,y \in \R^n 
\end{equation*}
and 
$\sup_{x \in \R^n}|A_l(x)|\leqslant M$ for some $M>0$.

\item [(3)] 
$A_{il} \in W_{\loc}^{1,2n}(\R^n)$.
\end{enumerate}\end{asu}

 Assumption \ref{N1} is essential for existence and uniqueness of the strong
solution of SDE (\ref{e1}), as well as the convergence of our approximating scheme. By Stroock-Varadhan theorem conditions (1) and (2), drawing from parabolic PDE theory on regularity of solutions,  assure the existence of a weak solution to SDE (\ref{e1}).
And the uniform H\"older continuity in (2) is crucial to derive some upper bound for the fundamental solution of 
parabolic PDE. Part (3) is the basic assumption, for the existence of a strong solution and pathwsie uniqueness of the elliptic SDE, in Veretennikov \cite{Ve}.  In fact in Watanabe-Yamada's celebrated paper \cite{Watanabe-Yamada71}, it was shown that pathwise uniqueness  holds for non-Lipschitz vector fields with regularity of the form 
$|A_l(x)-A_l(y)|\leqslant \rho_1(|x-y|), 1\leqslant l \leqslant m$, $|A_0(x)-A_0(y)|\leqslant \rho_2(|x-y|)$, for e.g. $\rho_1(t)=t\sqrt{|\log t|}$, $\rho_2(t)=t |\log t|$ when $t$ is small, essentially the same regularity required for the uniqueness of a  deterministic differential equation. See also a recent work by Fang-Zhang  \cite{FZ} for latest progress. When the SDE is uniformly elliptic this condition weakens as in the work of Veretennikov  \cite{Ve}. 
About the existence  and uniqueness of the strong solution of SDE (\ref{e1}), see also the work of Krylov-R\"okner \cite{KR} who discussed SDEs with additive noise with drift in $L^p$ and Flandoli-Gubinelli-Priola \cite{FGP} for SDE with $C_b^3$ diffusion coefficients and a
`locally uniformly $\alpha$-H\"older continuous' condition on the drift coefficient. See also Zvonkin and \cite{ZV} and Zhang \cite{Zh}.

Here in this paper we  take advantage of the elliptic system at the level of the derivative flow. 
 
 \vskip 5pt

{\bf Condition $G(\sigma, T_0)$. }
Let $G(x):=\sum_{l=0}^m |DA_l(x)|^2$.
There exist $\sigma >0$ and $T_0>0$ such that for all bounded set $S$\begin{equation*}%\label{N17}
\sup_{x \in S}\int_0^{T_0}\int_{\R^n}\e^{\sigma G(y)}K_s(x,y)dyds<\infty,
\end{equation*}
where $K_s(x,y)$ is the heat kernel on $\R^n$.
\vskip 10pt

For example if $G \in L^1_{\loc}(\R^n)$, $G(x)\leqslant C|x|^2$ for $x$ outside of a compact set, this condition holds on the time interval $[0,T]$ if  ${\sigma T}$ is sufficiently small. If the weak derivatives $DA_l(x)$ grows sub-linearly
 the integrability holds for all parameters. %For simplicity we employed bounds on $|A_0|$, but as can be seen in the text, one sided bounded on the drift vector field, $\sup_{|v|=1}\<DA_0(v),v\>$, is  sufficient. 
%An example of a function which satisfies $G(\sigma, t)$ for all $\sigma>0$ 
%and $t$ is $G(x)=f(|x|)$ where $f: \R_+\to \R$ is given by
%$$f(s)=s\log s\1_{s<e} +{s\over \log s}\1_{s\ge e}.$$
%If we allow singularities of the following form at $0$: $|\log s|^\alpha$ some $\alpha<1$, noticing that $\log s$ behaves like ${s}$ near the singularity, the resulting function  satisfies the integrability condition $G(\sigma,t)$ for all $\sigma$ and all $t$. Another possible function has the form  $$\sum_{{p\over q}\in [0,1], p,q\in {\mathbb Q}}{1\over q^3} f(x-{p\over q})$$
%for $s$ small, and $f$ is one of the functions defined earlier.

Consider an one-dimensional example ($n=1$), with $l=1$, $A_1(x)=1+\int_0^{x}\Big(\sqrt{\beta \big|\log |y|\big|}I_{|y|\leqslant 1}\Big)dy$
where $\beta>0$ is a positive constant, and $A_0(x)=0$. 
The SDE (\ref{e1}) with above coefficients satisfy 
Assumption \ref{N1} and Condition $G(\sigma, T_0)$ for all $0<\sigma<\frac{1}{\beta}$ and $T_0$. 
In fact, it is obvious that Assumption \ref{e1} holds and note that for any $T_0>0$, 
\begin{equation*}
\begin{split}
&\int_0^{T_0}\int_{\R^1}\e^{\sigma G(y)}K_s(x,y)dyds\leqslant
\int_0^{T_0}s^{-\frac{1}{2}}\Big(\int_{\R^1}\big(\frac{1}{|y|^{\sigma \beta}}
I_{|y|\leqslant 1}+ 1 I_{|y|>1}\big)\e^{-\frac{|x-y|^2}{2T_0}}dy \Big)ds
\end{split} 
\end{equation*}
so Condition $G(\sigma, T_0)$ holds if $\sigma<\frac{1}{\beta}$. 
 
% An example of a function which satisfies $G(\sigma, t)$ for all $\sigma>0$ 
%and $t$ is $G(x)=f(|x|)$ where $f: \R_+\to \R$ is given by
%$$f(s)=s\log s\1_{s<e} +{s\over \log s}\1_{s\ge e}.$$
%If we allow singularities of the following form at $0$: $|\log s|^\alpha$ some $\alpha<1$, noticing that $\log s$ behaves like ${s}$ near the singularity, the resulting function  satisfies the integrability condition $G(\sigma,t)$ for all $\sigma$ and all $t$. Another possible function has the form  $$\sum_{{p\over q}\in [0,1], p,q\in {\mathbb Q}}{1\over q^3} f(x-{p\over q})$$
%for $s$ small, and any $f$ which satisfies $G(\sigma,t)$ for all positive $\sigma$ and $t$.
%For simplicity we employed bounds on $|A_0|$, but as can be seen in the text, one sided bounded on the drift vector field, $\sup_{|v|=1}\<DA_0(v),v\>$, is  sufficient.

 The first result we state here is on the construction of a solution to the derivative SDE (\ref{e4}). 
We show that there is a regularising family of elliptic SDEs with parameter $\varepsilon$ such that the derivative flows $V_t^\varepsilon$ converge under a small time interval when $\varepsilon$ tends to $0$ , and the limit process is the unique 
solution of SDE (\ref{e4}) on this time interval. And from that we  construct 
a solution of SDE (\ref{e4}) in any time interval. In particular, the derivative in  
SDE (\ref{e4}) is the weak derivative. Under these conditions $A_l$ are 
not regular enough for us to obtain the required bounds directly  we employ the upper bound of the Markov kernel and the  integrable condition of $DA_l$ to estimate the moments of the derivative process.

\vskip 10pt {\bf Theorem \ref{Th2}. }
%Assume Assumption \ref{N1}  and $G(\sigma, T_0)$ holds for the function $G$ the functions $F_i$ from Remark \ref{remark-G}. 
Under  Assumption \ref{N1}  and condition  $G(\sigma, T_0)$, 
 there exists a process $V_t(x), 0\leqslant t < \infty$, such that 
%for $G=F_i$ in Remark \ref{remark-G}.  
for each  $p>0$, there is a $T_4>0$ as in Lemma \ref{Lem N3}, 
%for each $x \in \R^n$ and $t\leqslant \tilde T_0$, there is a process 
%$V_t(x,\varepsilon)$, such that for any bounded set $S$ in $\R^n$
\begin{equation*}
\lim_{\varepsilon\rightarrow 0}\sup_{x \in S}
\E\sup_{0\leqslant s \leqslant T_4}|V_s^{\varepsilon}(x)-V_s(x)|^p=0 
\end{equation*}
holds for any bounded set $S$ in $\R^n$. 
%And the process $V_t(x)$ is the unique strong solution of SDE (\ref{e4}) in time interval
%$0\leqslant t \leqslant \tilde T_0$. 
Furthermore, the process $V_t(x)$ is the unique strong solution
of SDE (\ref{e4}).
% in the time interval $0\leqslant t \leqslant T$ for any $T>0$.

 %is $G(x)=f(|x|)$ where $f: \R_+\to \R$ is given by
%$$f(s)=s\log s\1_{s<e} +{s\over \log s}\1_{s\ge e}.$$
%If we allow singularities of the following form at $0$: $|\log s|^\alpha$ some $\alpha<1$, noticing that $\log s$ behaves like ${s}$ near the singularity, the resulting function  satisfies the integrability condition $G(\sigma,t)$ for all $\sigma$ and all $t$. Another possible function has the form  $$\sum_{{p\over q}\in [0,1], p,q\in {\mathbb Q}}{1\over q^3} f(x-{p\over q})$$
%for $s$ small, and $f$ is one of the functions defined earlier.

\vskip 6pt

In the case of locally Lipschitz coefficients the result to compare with is that of Blagovescenskii-Friedlin  \cite{BF} , which goes back to 1961, where it is stated that if the 
coefficients are globally Lipschitz continuous, there exists a version of the solution which is jointly continuous in time and  space.  This result has been strengthened in terms of the growth on the derivative of the vector fields if all the vector fields are differentiable. See e.g.  Li \cite{Li1}, Fang \cite{Fang} for the cases about SDEs with locally Lipschitz continuous 
coefficients. See also Zhang \cite{Zh1} for the case
in which the coefficients are not Lipschitz continuous. 
In a recent work \cite{FGP}, Flandoli-Gubinelli-Priola study the case where diffusion coefficients are $C_b^3$, drift coefficients  are locally uniformly $\alpha$- H\"older continuous and obtain the existence
of a version of the solution which is $C^1$ with the space variable and a Bismut type formula. 
%In their methodology it appears that  the $C_b^3$ assumption on the diffusion vector fields are crucial. 

%However the price to pay are the regularity condition imposed on the vector fields. 
%Without ellipticity, for a typicl SDE, local Lipschitz condition or one in the spirit of Watanabe-Yamada is essential for the existence and pathwise uniqueness. The Watanabe-Yamada condition is 
%$|A_\ell(x)-A(y)|\le \rho(|x-y|)$ fro $\rho$ with $\int_{0+}{1\over \rho^2(t)}dt=\infty$ and $\rho^2(t)/t$.
%See also a recent work by Fang-Zhang  \cite{FZ} for latest progress. With ellipticity condition the picture changed completely, 
%see for example the work of Krylov-R\"okner \cite{KR} who discussed SDEs with additive noise with drift in $L^p$ and Flandoli-Gubinelli-Priola \cite{FGP} for significant work on strong completeness. See also  Zvonkin \cite{ZV} and Zhang \cite{Zh}. 

As for the continuous flow property, let  $\xi_t(x,\omega), 0\le t<\zeta(x, \omega)$ be its maximal solution starting from $x$ 
and  $\zeta(x)$ is the explosion time which we assume to be $\infty$ a.s. for each fixed $x \in \R^n$.  It is indicated that if $V_t(x)$ is a version of $D_x\xi_t(x)$, the derivative of $\xi_t(-, \omega)$ at point $x$, moment bounds on $V_t(x)$ relates to both completeness and strong completeness \cite{Li1,Li2}. For example non-explosion from particular starting point and the condition that $\sup_{x\in S}\E\sup_{s\leqslant t}  |V_s(x)|^p \chi_{t<\zeta(x)}$ for all  bounded set $S$ and some $p>n$ implies completeness from all initial points and the strong completeness , i.e. the existence 
of  a version of the solution which is jointly continuous in time and  space.

  The essential analysis on SDE's whose coefficients  are locally Lipschitz continuous are gathered in section  \ref{lip}.  Suppose that  Assumption \ref{J1} in  Section \ref{lip} holds and $A_l$ are elliptic there is a  smooth approximation for the  derivative process and the uniform convergence holds on any time interval. In this case a sequence of Lipschitz continuous cut-off functions are employed to 
approximate a Locally Lipschitz continuous system and we can remove the boundedness 
conditions on $A_l$ and $DA_l$.  
%If furthermore the coefficients are $C^1$ uniform ellipticity can be dispensed with.  

From this, for the SDE whose coefficients satisfy the conditions above, we obtain a representation for the derivative of the Markov semigroup associated with the SDE, the intertwining property of the differential $d$ and the semigroup $P_t$. 
Another application is that it can be shown that  under suitable conditions there is a continuous version of the solution, which is furthermore weakly differentiable and belongs to the Sobolev space $W_{loc}^{1,p}$ for some $p$ in small time interval. We also prove  
%solution of the SDE is differentiable in the sense of Malliavin calculus and 
an extrinsic integration by parts formula on path space.    

\vskip 10pt

 Standard investigation with regularity of stochastic flows assumes a local Lipschitz condition. The following result compliments known results, see e.g. Kunita \cite{KU}. We say that an SDE 
has a global continuous solution flow 
is strong complete if  for each starting point $x$ there is a global solution $(\xi_t(x), t\ge 0)$ and that there is a modification with $(t,x)\in [0,\infty)\times \R^n \to \xi_{\cdot}(\cdot,\omega)\in \R^n$  continuous almost surely.
 
\vskip 4pt 

%Under the Assumption \ref{N1} and Condition $G(\sigma, T_0)$  there
%is a global solution flow $\xi_{t}(x,\omega)$ for SDE (\ref{e1}), such that for almost surely all  $\omega$ 
%$\xi_{\cdot}(\cdot,\omega)$ is continuous in $[0,\infty)\times \R^n$. Furthermore  for each $p>0$ there is a constant $T=T(M,\theta,n, {T_0},p,\sigma)$ such that for $0\le t\le T$,
%$\xi_t(\cdot, \omega) \in W^{1,p}_{\loc}(\R^n)$. If condition $G(\sigma, T_0)$ holds for all $\sigma>0$ and all $T_0$ then  $\xi_t\in W^{1,p}$ for each $t\in \R_+$.
%the functions $F_i$ defined in Remark \ref{remark-G}. 
{\bf Theorem \ref{th:regularity}.}
Under Assumption \ref{N1} and Condition $G(\sigma, T_0)$,  
 SDE (\ref{e1}) has a global continuous solution flow.  Furthermore for each $p>0$ there is a constant $T_5$, 
%(K,\alpha, M,\theta,n,p, T_0,\tilde{T}, \sigma)$ 
such that 
$\xi_t(\cdot, \omega) \in W^{1,p}_{\loc}(\R^n)$ for each $t\in [0,T_5]$.

\vskip 6pt

%Under Assumption \ref{N1} and Condition $G(\sigma, T_0)$ 
% there is a constant $T=T(M,\theta,n,T_0,\sigma)$ such that
%\begin{equation}\label{e23}
%d P_t f(x)(v_0)
%=\frac{1}{t}\E \big[f(\xi_t(x))\int_{0}^{t}
%\langle Y(\xi_s(x))(V_s(x,v_0)), dB_s\rangle_{\R^m} \big], \ \ v_0\in \R^n
%\end{equation} 
%for any $f$ in $ {\mathcal B}_b(\R^n)$ and $0<t\leqslant T$. If moroever 
%$f \in C_b^1(\R^n)$, the intertwining formula holds for such $t$:
% $d(P_tf)(x)(v_0)=\E df(V_t(x,v_0))$,  $x, v_0 \in \R^n$.
{\bf {Theorem \ref{P1}}}
Suppose the Assumption \ref{N1} and condition $G(\sigma, T_0)$  hold, 
% Condition $G(\sigma,T_0)$ holds for the functions $F_i$, 
then there is a constant $T_{6}$, %(K,\alpha, M,\theta,n,T_0,\tilde{T},\sigma)$, 
 such that
for each  $t \in [0,T_6]$,
\begin{equation*}%\label{e23}
d P_t f(x)(v_0)
=\frac{1}{t}\E \big[f(\xi_t(x))\int_{0}^{t}
\langle Y(\xi_s(x))(V_s(x,v_0)), dW_s\rangle_{\R^m} \big], \ \ v_0\in \R^n
\end{equation*} 
holds for all $f$ in $ {\mathcal B}_b(\R^n)$. If moreover 
$f \in C_b^1(\R^n)$, for all $v_0\in \R^n$ and $t\in [0, T_6]$,
 $$d(P_tf)(x)(v_0)=\E df(V_t(x,v_0)).$$

\vskip 10pt

Finally we have the following integration by parts formula . 
Let $C_{x}([0,t];\R^n)$ be the space of  continuous functions from $[0,t]$ to $\R^n$ with
initial value $x$.

%Under Assumption \ref{N1} and Condition $G(\sigma, T_0)$ there is a
%constant $\bar T>0$, such that for any $0\leqslant T\leqslant \bar T$ the following holds.
%If $h$ is an adapted stochastic process,  $h:[0,T]\times\Omega\rightarrow \R^n$,
%with  $h(\omega)\in L_0^{2,1}([0,T]; \R^m)$ for almost surely all $\omega$ and such that $\E(\int_0^T
%|\dot{h}_s|^2ds)^{\frac{1+\beta}{2}}<\infty$ for some $\beta>0$, then $$\E dF(V^h(\xi_{\cdot}))=\E F(\xi_{.}(x))\delta V_{T}^h(\xi_{\cdot}) $$
%where $F$ is the smooth cyclindrical function on path space
%$C_{x}([0,T]; \R^n)$.

\vskip 5pt

{\bf Theorem \ref{ibp}. } 
Assume Assumption \ref{N1} and  Condition $G(\sigma, T_0)$. There is a
positive constant $T_{8}$, such that for any  $t\in (0, T_{8}]$  the following integration by parts formula 
holds for every
$BC^1$ function $F$ on $C_{x}([0,t];\R^n)$.
Let  $h:[0,t]\times\Omega\rightarrow \R^n$ be an adapted stochastic process with $h(\cdot, \omega)\in L_0^{2,1}([0,t]; \R^m)$ almost surely and  $\E(\int_0^{t}
|\dot{h}_s|^2ds)^{\frac{1+\beta}{2}}<\infty$ for some $\beta>0$. Then
$$\E dF(V^h(\xi_{\cdot}))=\E F(\xi_{\cdot}(x))\delta V_{t}^h(\xi_{\cdot}) $$
for $\delta V^h$  defined by (\ref{divergence}).

\vskip 10pt

Finally in the Appendix we analyse the non-smooth geometry induced by the SDE in terms of an approximation linear connection when $\R^n$ is treated as a manifold.  We overcome the difficulty that the limiting connection may not be torsion skew symmetric with respect to the relevant induced Riemannian metric
and  obtain an intrinsic integration by part formula. 

%In the non-smooth case we discuss a smooth approximation which preserves much of the properties of the connection, which leads to a non-smooth Riemannian geometry. 
\section{An approximation scheme for the derivative flow}
\label{section2}

In this section we consider  a family of   SDEs whose coefficients are smooth,
\begin{equation} \label{e2-1}
d\xi_t^{\varepsilon}(x)=\sum_{l=1}^{m} A_l^{\varepsilon}(\xi_s^{\varepsilon}(x))dW_s^l
+A_0^{\varepsilon}(\xi_s^{\varepsilon}(x))ds
\end{equation}
with the property that if the solutions $\xi_t^{\varepsilon}(x)$ and $\xi_t(x)$ are the solution of SDE (\ref{e2-1}) and (\ref{e1}) respectively with the same starting point $x$, there is a convergence theorem.
%$$\lim_{\varepsilon \rightarrow 0} \sup_{|x|\leqslant R}
%\E\big[\sup_{0\leqslant s \leqslant T}|\xi_s^{\varepsilon}(x)-\xi_s(x)|^p\big]=0.$$
%In the sequel we call such a family of SDE's an {\it an approximating family of SDEs}
%and unless otherwise stated the approximating family is as constructed below.

Let $\eta:\R^n\rightarrow R$ be the smooth mollifier defined by 
$\eta(x)=C\e^{\frac{1}{|x|^2-1}}\1_{ |x|<1}$
   where $C$ is a normalising constant such that $|\eta|_{L^1}=1$. 
Define a sequence of smooth functions $\eta_\varepsilon$ with support in the ball $B_\varepsilon$, of radius $\varepsilon$ centred at $0$ by  
$\eta_{\varepsilon}(x)=\varepsilon^{-n}\eta(\frac{x}{\varepsilon})$.
 For a locally integrable function $f$ on $\R^n$ let $f_{\varepsilon}$ be its convolution  with $\eta_{\varepsilon}$, 
\begin{equation}\label{e0}
\begin{split}
f_{\varepsilon}(x)=\eta_{\varepsilon}\ast f(x)=\int_{\R^n}\eta_{\varepsilon}(x-y)f(y)dy
=\int_{B_{\varepsilon}(0)}\eta_{\varepsilon}(y)f(x-y)dy 
\end{split}
\end{equation}
Then $f_\varepsilon\to f$ for almost surely all $x$ and the approximation family are uniformly Lipschitz continuous if the original functions are and have uniform linear growth if the original function does. To summarise, we have the following lemma. (see 
\cite{Evans})

 \begin{lemma}\label{lemma1}
 %\label{lemma n1}
\begin{enumerate}
\item[(1)]  If $f:\R^n\rightarrow R$ such that $|f(x)|\leqslant \psi(|x|)$ for $\psi$ a positive  increasing function, then
 $|f_{\varepsilon}(x)|\leqslant \psi (|x|+\varepsilon)$. 
%and $|f_{\varepsilon}(x)-f(x)|\leqslant \psi(|x|)\varepsilon $.
If $f$ is uniformly H\"older continuous in $\R^n$ i.e. 
$|f(x)-f(y)|\leqslant K|x-y|^{\alpha}$ for some $K>0,\ 0<\alpha<1$, we have 
$$\sup_{\varepsilon}|f_{\varepsilon}(x)-f_{\varepsilon}(y)|\leqslant K|x-y|^{\alpha} 
\ \ \ \forall x,y \in \R^n$$
and 
$$\lim_{\varepsilon\rightarrow 0}\sup_{x \in \R^n}|f_{\varepsilon}(x)-f(x)|=0 $$
If $f$ is locally Lipschitz with rate function $K$ then  so is each $f_\varepsilon$ with the same rate function,
that is $|f_{\varepsilon}(x)-f_{\varepsilon}(y)|\leqslant  K(n+\varepsilon) |x-y| $, for all  $x,y\in B_n$.
\item[(2)]
 If $f \in L^p_{\loc}(\R^n)$, then for any $R>0$, we have
\begin{equation*}
\int_{|x|\leqslant R} f_{\varepsilon}^pdx\leqslant 
\int_{|x|\leqslant R+\varepsilon} f^pdx
\end{equation*}
$$\lim_{\varepsilon\rightarrow 0}\int_{|x|\leqslant R}|f_{\varepsilon}-f|^p dx=0 $$
\end{enumerate}
\end{lemma}

%   is a bounded Lipschitz continuous function with bound $M$ and Lipschitz 
%constant $K$, and 
% for each $\varepsilon>0$ and  $x, y \in \R^n$,\\
% (1) $|f_{\varepsilon}(x)|\leqslant M$\\
%(2) $|f_{\varepsilon}(x)-f(x)|\leqslant K\varepsilon $ \\
%(3) $|f_{\varepsilon}(x)-f_{\varepsilon}(y)|\leqslant K|x-y| $

%For bounded Lipschitz continuous vector fields $A_l(x)=(\sigma_{l1}(x), 
%\sigma_{l2}(x)...\sigma_{ln}(x))$, 
\subsection{Basic Estimates}
Unless otherwise stated we take  $A_{li}^{\varepsilon}=\eta_{\varepsilon}\ast A_{li}$, 
$A_l^{\varepsilon}=(A_{1l}^{\varepsilon}, \dots ,A_{nl}^{\varepsilon})^T$. Let  
 $A^\varepsilon=(A_1^\varepsilon, \dots, A_m^\varepsilon)$ 
 and $a_{ij}^\varepsilon=\sum_{l=1}^m A_{il}^\varepsilon(x)A_{jl}^\varepsilon(x)$.
We begin with  a family of  approximating  SDEs, with the smooth coefficients $A_l^{\varepsilon}$ defined 
as above,
\begin{equation} \label{e2}
d\xi_t^{\varepsilon}(x)=x+\sum_{l=1}^{m}\int_0^t A_l^{\varepsilon}(\xi_s^{\varepsilon}(x))dW_s^l
+\int_0^tA_0^{\varepsilon}(\xi_s^{\varepsilon}(x))ds
\end{equation}
 We summarise below useful property of this approximation.
\begin{lemma}
\label{lemma2} 
Suppose Assumption 2.1 holds, then for some $\varepsilon_0>0$,  $\{a_{ij}^\varepsilon, \varepsilon<\varepsilon_0\}$ are elliptic with the same uniform elliptic constant.
\end{lemma}
\begin{proof}
%Since $a_{ij}^{\varepsilon}(x):=\sum_{l=1}^{m}\sigma^{\varepsilon}_{li}(x) \sigma^{\varepsilon}_{lj}(x)$
%Assume that the coefficients of the SDE have bound $M$,  Lipschitz constant $K$ and elliptic constant $\theta$.
By Lemma \ref{lemma1}, 
$A^{\varepsilon}_{il}(x)$ are uniformly bounded in all parameters and 
the family of functions $A^{\varepsilon}_{il}(x)$ converge uniformly in $\R^n$ as $\varepsilon$
tends to 0.
Let  $\varepsilon_0$ be such that $\varepsilon< \varepsilon_0$, 
$\sup_{x \in \R^n}|a_{ij}^{\varepsilon}(x)-a_{ij}(x)|\leqslant \frac{\theta}{2n}$.
Hence for $x\in \R^n, \xi=(\xi_1,\xi_2...,\xi_n) \in \R^n$,
\begin{equation*}
\begin{split}
\sum_{i,j=1}^{n}a_{ij}^{\varepsilon}(x)\xi_i \xi_j &\geqslant
\sum_{i,j=1}^{n}a_{ij}(x)\xi_i \xi_j-\sum_{i,j=1}^{n}|a_{ij}^{\varepsilon}(x)-a_{ij}(x)||\xi_i|| \xi_j|\\
&\geqslant\theta \sum_{i=1}^n |\xi_i|^2-\frac{\theta}{2n}n\sum_{i=1}^n |\xi_i|^2
=\frac{\theta}{2} \sum_{i=1}^n |\xi_i|^2
\end{split}
\end{equation*}
\end{proof}

 From Theorem A in \cite{Kan}, we obtain
the following results on the approximation of $\xi_t^{\varepsilon}(x)$ and $\xi_t(x)$.

%be the $n\times m$ matrix  obtained by the  convolution procedure,  and let  

\begin{lemma}\label{N2}
If assumption \ref{N1} holds, for any $p>0,R>0,T>0$, 
\begin{equation}\label{N35}
\lim_{\varepsilon \rightarrow 0} \sup_{|x|\leqslant R}
\E\big[\sup_{0\leqslant s \leqslant T}|\xi_s^{\varepsilon}(x)-\xi_s(x)|^p\big]=0,
\end{equation}
\end{lemma}

\begin{proof}
As indicated in the introduction,  Assumption \ref{N1} implies that there is  a unique strong solution for SDE (\ref{e1}), see Theorem 1 in \cite{Ve}.  Under the assumptions of the theorem,  we may apply Lemma \ref{lemma1} and Theorem A  in \cite{Kan} to obtain (\ref{N35}) for $p=2$. The case of $p<2$ follows from the H\"older's inequality.
If $p>2$, we have for each $T>0$,
\begin{equation}\label{N36-}
\begin{split}
&\E\sup_{0\leqslant s \leqslant T}|\xi_s^{\varepsilon}(x)-\xi_s(x)|^p
=\E \Big[\big(\sup_{0\leqslant s \leqslant T}|\xi_s^{\varepsilon}(x)-\xi_s(x)|\big)
\big(\sup_{0\leqslant s \leqslant T}|\xi_s^{\varepsilon}(x)-\xi_s(x)|\big)^{p-1}\Big]\\
&\leqslant \sqrt{\E\big[\sup_{0\leqslant s \leqslant T}|\xi_s^{\varepsilon}(x)-\xi_s(x)|^2\big]}
\sqrt{\E\big[\sup_{0\leqslant s \leqslant T}|\xi_s^{\varepsilon}(x)-\xi_s(x)|^{2p-2}\big]}
\end{split} 
\end{equation}
 In fact, from the uniform boundedness of  the coefficiens  of SDE (\ref{e1}) and that of (\ref{e2}), 
we can get the following for all  bounded set $S$ in $\R^n$, $T>0,p>0$ 
\begin{equation}\label{N37}
\sup_{0<\varepsilon<\varepsilon_0} \sup_{x \in S} \E\sup_{0\leqslant s \leqslant T}|\xi_s^{\varepsilon}(x)|^{2p-2}+
 \sup_{x \in S} \E\sup_{0\leqslant s \leqslant T}|\xi_s(x)|^{2p-2} <\infty.
\end{equation}
So the conclusion follows from (\ref{N36-}) and (\ref{N37}).
\end{proof}

\begin{lemma}\label{Lem N2}
Suppose $(\xi_s^{\varepsilon}(x), s\geqslant 0)$ are the solutions to the family of smooth SDEs (\ref{e2-1}). Assume that 
the coefficients of these SDEs are uniformly elliptic with a common uniform elliptic constant $\theta$ and 
uniformly bounded with a common bound $M$, and
$\sup_{\varepsilon}|A_l^{\varepsilon}(x)-A_l^{\varepsilon}(y)|\leqslant K|x-y|^{\alpha} 
\ \forall x,y \in \R^n$ for some $K>0, \ 0<\alpha <1$. For each $(s,x)$, we assume that $\xi_s^{\varepsilon}(x)$ 
converges to $\xi_s(x)$  almost surely as $\varepsilon$ tends to zero.  Then the distribution of  $\xi_s(x)$  is absolutely continuous with respect to the Lebesgue measure in $\R^n$, 
% i.e.  there exists a positive function $p_s(x,.): \R^n\rightarrow \R$, such that for 
%each bounded measurable function $f$ in $\R^n$, 
%$$ \E f(\xi_s(x))=\int_{\R^n}p_s(x,y)f(y)dy $$
and we have the following estimate for the transition kernel $p_s(x,.)$, 
 \begin{equation}\label{N3}
p_s(x,y)\leqslant  C_1 s^{-\frac{n}{2}}\e^{-\frac{|x-y|^2}{2C_1 s}}, \qquad   \forall\; s\in(0,T],   \; x,y \in \R^n 
\end{equation}
where the constant $C_1$  depends only on $K,\alpha,M,\theta, n, T$.
In particular, under the Assumption \ref{N1}, the estimate holds for the Markov kernel of  SDE (\ref{e1}). 
\end{lemma}

\begin{proof}
Denote by ${\mathfrak B}_b(\R^n)$ the set of bounded measurable functions in $\R^n$. Since 
the coefficients of SDE (\ref{e2-1}) are smooth and  $A_l^{\varepsilon}, 1\leqslant l \leqslant m$
are uniformly elliptic, the distribution of the solution $\xi_t^{\varepsilon}(x)$   is 
absolutely continuous with respect to the Lebesgue measure in $\R^n$ for each $s>0, \varepsilon>0$.
(for example, see \cite{Nua}).  
Let $p_s^{\varepsilon}(x,y):\R^n\times \R^n\rightarrow \R$ be the Markov kernel,  so 
  for $f\in {\mathfrak B}_b(\R^n)$,
\begin{equation*}
\E f(\xi_s^{\varepsilon}(x))=\int_{\R^n}p_s^{\varepsilon}(x,y) f(y) dy. 
\end{equation*}
By  classical results in diffusion theory, $p_s^{\varepsilon}(x,y)$ is the fundamental
solution of the following parabolic PDE, 
\begin{equation*}
\begin{cases}
& \frac{\partial u^{\varepsilon}}{\partial t}=\sum_{i,j}a_{ij}^{\varepsilon}\frac{\partial u^{\varepsilon}}{\partial x_i \partial x_j}+
\sum_i A_{i0}^{\varepsilon} \frac{\partial u^{\varepsilon}}{\partial x_i}\\
& u^{\varepsilon}(0,y)=\delta_{x}. 
\end{cases}
\end{equation*}
By the estimate for the fundamental solution of non-divergence form 
parabolic PDE, see  \cite{Fri} or \cite{LKO}, there are constants $c_1, c_2$ such that for $s\in (0,T]$,
\begin{equation*}
 p_s^{\varepsilon}(x,y)\leqslant c_1 s^{-\frac{n}{2}}\e^{-\frac{|x-y|^2}{2c_2 s}}, \qquad  s\in (0,  T],   \; \varepsilon\in (0, \varepsilon_0]. 
\end{equation*}
 The constants depend only on  the uniform elliptic constants of $a_{ij}^{\varepsilon}$, the bounds on $a_{ij}^{\varepsilon}$ and $A_{i0}^{\varepsilon}$, the H\"older constants $K,\alpha$ of $a_{ij}^{\varepsilon}, A_{i0}^{\varepsilon}$, 
the dimension $n$ and time interval $T$. In particular the constants are independent of $\varepsilon$ when $\varepsilon$ is sufficiently small by the condition of this lemma. Take $C_1=\max(c_1,c_2)$ for simplicity. 

% So for each $T>0$, there are constants $\varepsilon_0>0$ and $c_1, c_2$ which 
%only depend on $M,\theta, n,T$ s.t.

As assumed in the condition, $\lim_{\varepsilon\to 0}\xi_s^\varepsilon(x)=\xi_s(x)$ almost surely.  Taking $\varepsilon\to 0$, by the Lebesgue  dominated convergence theorem, 
for each $f \in C_b(\R^n)$,\begin{equation}\label{N4}
\E f(\xi_s(x))\leqslant C_1 s^{-\frac{n}{2}}\int_{\R^n} \e^{-\frac{|x-y|^2}{2C_1 s}}f(y) dy 
\end{equation}

For every bounded open sets $O$ in $\R^n$, let $f_k$ be a sequence of non-negative functions in
$ C_b(\R^n)$ such that
$\lim_{k\rightarrow \infty} f_k=I_{O}$ where the convergence is pointwise.  Fatou lemma leads to 
\begin{equation}\label{N5}
\E I_{O}(\xi_s(x))\leqslant C_1 s^{-\frac{n}{2}}\int_{O} \e^{-\frac{|x-y|^2}{2C_1 s}}f(y) dy 
\end{equation}
The same inequality also holds for every open, not necessarily bounded set by Fatou lemma. It follows that if 
$\Gamma$ in $\R^n$ is a set with $\lambda(\Gamma)=0$, where $\lambda$ denotes the Lebesgue, then $\E I_{\Gamma}(\xi_s(x))=0$. 
In fact, by regularity of the measure, there are open subsets  $O_k$ 
with $\Gamma\subseteq O_k,\  \lambda(O_k)\downarrow 0$ and so
\begin{equation*}
\E I_{\Gamma}(\xi_s(x))\leqslant \E I_{O_k}(\xi_s(x))\leqslant C_1 s^{-\frac{n}{2}}\int_{O_k} \e^{-\frac{|x-y|^2}{2C_1 s}}f(y) dy  \stackrel{k\to \infty}{\to} 0.
\end{equation*}
 Now  the distribution of  $\xi_s(x)$  is 
absolutely continuous with respect to the Lebesgue measure in $\R^n$  for each
$x \in \R^n$ and by (\ref{N4}) for all $f \in C_b(\R^n)$, we have,  
\begin{equation*}
\int_{\R^n}p_s(x,y)f(y)dy \leqslant C_1 s^{-\frac{n}{2}}\int_{\R^n} \e^{-\frac{|x-y|^2}{2C_1 s}}f(y) dy  
\end{equation*}
By general approximation procedure the above inequality also holds for each 
$f \in {\mathfrak B}_b(\R^n)$, which implies the 
Markov kernel $p_s(x,y)$ has the same upper bound: $C_1 s^{-\frac{n}{2}}\e^{-\frac{|x-y|^2}{2C_1 s}} $. 
If Assumption \ref{N1} holds, by Lemma \ref{lemma1}, \ref{lemma2} and Lemma \ref{N2}, all the condition above are satisfied  so  the upper bounded for the Markov kernel of the solution 
of SDE (\ref{e1}) follows. 
\end{proof}

Let $K_s(x,y):=s^{-\frac{n}{2}}\e^{-\frac{|x-y|^2}{2s}}, s>0, x,y \in \R^n$.
Let $g:\R^n\to \R^+$ be a Borel measurable function.  If there exists a constant $T_0>0$, such that $\int_0^{T_0}\int_{\R^n}g(y)K_s(x,y)dyds<\infty$, then by Lemma \ref{Lem N2}, 
%,   ??? the change of variable(for time variable) and Fatou lemma, 
$\int_0^{\min(T_0,\frac{T_0}{C_1})}\E g(\xi_s(x))ds<\infty$. 

From now on we assume 
%that $C\ge 1$ for simplicity. 
the estimate for the Markov kernel in Lemma \ref{Lem N2} are considered in time interval
$0<s \leqslant \tilde{T}$ for some fixed $\tilde{T}$

\begin{lemma}\label{lemma3-1}
Let $g:\R^n\to \R$ be a Borel measurable function. Assume Assumption \ref{N1}  and that there exist $T_0>0$, and 
$p\geqslant1$, such that 
\begin{equation}\label{N7-1}
\sup_{x \in S}\int_0^{T_0}\int_{\R^n}|g(y)|^{p}K_s(x,y)dyds<\infty 
\end{equation} 
for any bounded set $S$ in $\R^n$. Set $T_1=\min(\tilde{T},\frac{T_0}{C_1})$ where $C_1$ is the constant in the 
transition kernel, c.f. (\ref{N3}), on 
the time interval $(0,\tilde{T}]$. Then for  $g_{\varepsilon}=\eta_{\varepsilon}\ast g$ 
and any bounded set $S$ in $\R^n$, 
\begin{equation}\label{N36}
\sup_{\varepsilon<\varepsilon_0}\sup_{x \in S}
\int_0^{T_1}\E|g_{\varepsilon}(\xi_s^{\varepsilon}(x))|^{p}ds<\infty 
\end{equation}
where $\varepsilon_0$ is the constant  in Lemma \ref{lemma2}.
\end{lemma}
\begin{proof}
Recall the transition kernel estimates we use before, \begin{equation*}
 p_s^{\varepsilon}(x,y)\leqslant C_1 s^{-\frac{n}{2}}\e^{-\frac{|x-y|^2}{2C_1 s}}, \qquad  \forall \;
 s\in (0,\tilde{T}),  \varepsilon\in (0, \varepsilon_0). 
\end{equation*}
In the remaining part of the proof, the constants $C$ which appear in the computation 
may change from line to line and depend only on $K,\alpha, M,\theta,\delta,n,\tilde{T},p$.  Define $\tilde{K}_s(x)=K_s(x,0)$, $s>0$, $x\in \R^n$.
For  $T\in (0,\tilde{T}]$ and $\varepsilon\in (0,\varepsilon_0)$, we  derive the following estimate:
\begin{equation*}
\begin{split}
\int_0^{T}\E|g_{\varepsilon}(\xi_s^{\varepsilon}(x))|^{p}ds
&\leqslant C\int_0^{{C_1 T}} \int_{\R^n}|g_{\varepsilon}(y)|^{p} K_s(x,y)dyds\\
&\leqslant C \int_0^{C_1 T}\int_{\R^n}\eta_{\varepsilon}\ast |g|^{p}(y)
\tilde{K}_s(x-y)dyds\\
&=C \int_0^{C_1 T}\big(\eta_{\varepsilon}\ast |g|^{p}\big)\ast \tilde{K}_s(x)ds\\
&= C \Big(\int_0^{C_1 T}|g|^{p}\ast \tilde{K}_sds\Big)\ast \eta_{\varepsilon}(x)
\end{split}
\end{equation*}
The last step  is due to the property that $f\ast h=h\ast f $ for locally integrable 
functions and  Fubini's Theorem. Since we assume that
$\int_0^{T_0}|g|^{p}\ast \tilde{K}_s\; ds$ is locally bounded in $\R^n$
for any bounded set $S$ in $\R^n$, when $C_1 T \leqslant T_0$, i.e.
$T\leqslant T_1:=\min(\tilde{T},\frac{T_0}{C_1})$, the following holds:
  \begin{equation}\label{N01}
\sup_{\varepsilon<\varepsilon_0}\sup_{x \in S}
\int_0^{T}\E|g_{\varepsilon}(\xi_s^{\varepsilon}(x))|^{p}ds<\infty. 
\end{equation} 
\vskip 2pt 
\end{proof}

\begin{lemma} \label{lemma3}
Let $g:\R^n\to \R$ be a Borel measurable function. Suppose Assumption \ref{N1} holds and that there exist $T_0>0$, $\delta\in (0,1)$ and $p\ge 1$ such that 
\begin{equation}\label{N7}
\sup_{x \in S}\int_0^{T_0}\int_{\R^n}|g(y)|^{p(1+\delta/2)}K_s(x,y)dyds<\infty 
\end{equation}
for any bounded set $S$ in $\R^n$. 
If moroever, 
\begin{equation}\label{N6}
g \in L^{\overline{p}(n)}_{\loc}(\R^n) 
\end{equation}
where $\overline{p}(n)=\max\{p(1+\delta), \frac{pn(1+\delta)}{2}\}$. Let $T_1=\min(\tilde{T},\frac{T_0}{C_1})$
(constant $C_1$ is the same as that in Lemma \ref{lemma3-1}), 
%where $C_1$ is the constant in the heat kernel, c.f. (\ref{N3}), on 
%the time interval $(0,\tilde{T}]$. 
%Then for any bounded set $S$ in $\R^n$, 
%\begin{equation}\label{N36}
%\sup_{\varepsilon<\varepsilon_0}\sup_{x \in S}
%\int_0^{T_1}\E|g_{\varepsilon}(\xi_s^{\varepsilon}(x))|^{p(1+\delta/2)}ds<\infty 
%\end{equation}
then for $g_{\varepsilon}=\eta_{\varepsilon}\ast g$
and any bounded set $S$ in $\R^n$
\begin{equation}\label{N31}
\lim_{\varepsilon\rightarrow 0}\sup_{x \in S}
\int_0^{T_1} \E|g_{\varepsilon}(\xi_s^{\varepsilon}(x))-g(\xi_s(x))|^p ds=0. 
\end{equation}
\end{lemma}
\begin{proof}
%Recall Aronson's estimates we use before, \begin{equation*}
% p_s^{\varepsilon}(x,y)\leqslant C_1 s^{-\frac{n}{2}}\e^{-\frac{|x-y|^2}{C_1 s}}, \qquad  \forall \;
% s\in (0,\tilde{T}),  \varepsilon\in (0, \varepsilon_0). 
%\end{equation*}
%In the remaining part of the proof, the constants $C$ which appear in the computation 
%may change from line to line and depend only on $M,\theta,\delta,n,\tilde{T},p$. They do not depend on $\varepsilon$ and $R$, which is essential for taking $\varepsilon$ to $0$ and $R$ to infinity.  Define $\tilde{K}_s(x)=K_s(x,0)$, $s>0$, $x\in \R^n$.
%For  $T\in (0,\tilde{T}]$ and $\varepsilon\in (0,\varepsilon_0)$, we  derive the following estimate:
%\begin{equation*}
%\begin{split}
%\int_0^{T}\E|g_{\varepsilon}(\xi_s^{\varepsilon}(x))|^{p(1+\delta/2)}ds
%&\leqslant C\int_0^{{C_1 T}} \int_{\R^n}|g_{\varepsilon}(y)|^{p(1+\delta/2)} K_s(x,y)dyds\\
%&\leqslant C \int_0^{C_1 T}\int_{\R^n}\eta_{\varepsilon}\ast |g|^{p(1+\delta/2)}(y)
%\tilde{K}_s(x-y)dyds\\
%&=C \int_0^{C_1 T}\big(\eta_{\varepsilon}\ast |g|^{p(1+\delta/2)}\big)\ast \tilde{K}_s(x)ds\\
%&= C \Big(\int_0^{C_1 T}|g|^{p(1+\delta/2)}\ast \tilde{K}_sds\Big)\ast \eta_{\varepsilon}(x)
%\end{split}
%\end{equation*}
%The last step  is due to the property that $f\ast h=h\ast f $ for locally integrable 
%functions and  Fubini's Theorem. Since we assume that
%$\int_0^{T_0}|g|^{p(1+\delta/2)}\ast \tilde{K}_s\; ds$ is locally bounded in $\R^n$
%for any bounded set $S$ in $\R^n$, when $C_1 T \leqslant T_0$, i.e.
%$T\leqslant T_1:=\min(\tilde{T},\frac{T_0}{C_1})$, the following holds:
In the remaining part of the proof, the constants $C$ which appear in the computation 
may change from line to line and depend only on $K,\alpha, M,\theta,$
$\delta,n,\tilde{T},p$. They do not depend on $\varepsilon$ and $R$, which is essential for taking $\varepsilon$ to $0$ and $R$ to infinity.
%From Lemma \ref{lemma3-1}, condition (\ref{N7}) implies that, 
%  \begin{equation}\label{N01}
%\sup_{\varepsilon<\varepsilon_0}\sup_{x \in S}
%\int_0^{T_1}\E|g_{\varepsilon}(\xi_s^{\varepsilon}(x))|^{p(1+\delta/2)}ds<\infty. 
%\end{equation} 
\vskip 2pt

%For $S=\{x\}$, a singleton set,  (\ref{N31}) follows from the uniform integrability we just showed.
The required uniform convergence requires an uniform estimate. Since $W^{1,p}$ spaces are not included in $W^{1,\infty}$, we must sacrifice some integrability for this uniform estimate.
Fixed $T\leqslant T_1$, let 
\begin{eqnarray*}
&&\int_0^{T} \E|g_{\varepsilon}(\xi_s^{\varepsilon}(x))-g(\xi_s(x))|^p ds \\
&&\leqslant C \int_0^{T} \E|g_{\varepsilon}(\xi_s^{\varepsilon}(x))-g(\xi_s^{\varepsilon}(x))|^p ds
+ C\int_0^{T} \E|g(\xi_s^{\varepsilon}(x))-g(\xi_s(x))|^p ds\\
&&:=I_1^{\varepsilon}(x,T)+I_2^{\varepsilon}(x,T)
\end{eqnarray*}
Note that 

\begin{eqnarray*}
I_1^{\varepsilon}(x,T)
&\leqslant &C\int_0^T \E\big[|g^{\varepsilon}(\xi_s^{\varepsilon}(x))-g(\xi_s^{\varepsilon}(x))|^p
I_{\{|\xi_s^{\varepsilon}(x)|\leqslant R\}}\big]ds\\
&&+ C\int_0^T \E\big[|g^{\varepsilon}(\xi_s^{\varepsilon}(x))-g(\xi_s^{\varepsilon}(x))|^p
I_{\{|\xi_s^{\varepsilon}(x)|> R\}}\big]ds\\
&:=&I^{\varepsilon}_{11}(x,T,R)+I^{\varepsilon}_{12}(x,T,R).
\end{eqnarray*}

For each $\frac{1}{p_1}+\frac{1}{q_1}=1$, by Markov kernel estimate and H\"older inequality
\begin{eqnarray*}
&&I^{\varepsilon}_{11}(x,T,R)\\
&\leqslant& C\int_{0}^{C_1 T}
C s^{-\frac{n}{2}}\big(\int_{\R^n}\e^{-\frac{p_1|x-y|^2}{2s}}\big)^{\frac{1}{p_1}}
\big(\int_{|y|\leqslant R}|g_{\varepsilon}-g|^{pq_1}(y)dy  \big)^{\frac{1}{q_1}}ds\\
&\leqslant& C 
\int_{0}^{C_1 T}s^{-\frac{n}{2}(1-\frac{1}{p_1})}
\left(\int_{|y|\leqslant R} |g_{\varepsilon}-g|^{pq_1}(y)dy \right)^{\frac{1}{q_1}}ds\\
\end{eqnarray*}

When $n>1$,  we take $q_1=\frac{n(1+\delta)}{2}$ in above inequality. Then
\begin{equation}\label{N10}
\begin{split}
I^{\varepsilon}_{11}(x,T,R)
&\leqslant C 
\int_{0}^{C_1T} s^{-\frac{1}{1+\delta}}
\big(\int_{|y|\leqslant R} |g_{\varepsilon}-g|^{\frac{pn(1+\delta)}{2}}(y)dy \big)^{\frac{2}{n(1+\delta)}}ds\\
&\leqslant C \left(\int_{|y|\leqslant R} |g_{\varepsilon}-g|^{\frac{np(1+\delta)}{2}}(y)dy \right)^{\frac{2}{n(1+\delta)}}.
\end{split}
\end{equation}

For $n=1$ the corresponding estimate is
\begin{equation*}
I^{\varepsilon}_{11}(x,T,R)\leqslant C \big(\int_{|y|\leqslant R} |g_{\varepsilon}-g|^{p(1+\delta)}(y)dy \big)^{\frac{1}{(1+\delta)}}
\end{equation*}
%due to that $s^{ {1\over 2}+{1\over 2p_1}}$ is integrable in $s$.
Under  condition (\ref{N6}), by Lemma \ref{lemma1}, for each fixed $R>0$, we have 
\begin{equation}\label{N02}
\lim_{\varepsilon \rightarrow 0}\sup_{x \in S} I^{\varepsilon}_{11}(x,T,R)=0.
\end{equation}

For the term involving large $|\xi_s^\varepsilon(x)|$,  H\"older inequality gives
\begin{equation*}%\label{N11}
I^{\varepsilon}_{12}(x,T,R)\leqslant C \left(\int_0^T 
\E|g^{\varepsilon}(\xi_s^{\varepsilon}(x))-g(\xi_s^{\varepsilon}(x))|^{p(1+\delta/2)}ds\right)^{\frac{2}{2+\delta}}
\left(\int_0^T 
\frac{\E|\xi_s^{\varepsilon}(x)|^2}{R^2}\right)^{\frac{\delta}{2+\delta}}.\\
\end{equation*}

The first factor on the right hand side is bounded uniformly in $S$, since by 
(\ref{N7}) and Lemma \ref{lemma3-1}, 
\begin{equation}\label{G12}
\sup_{x \in S}\int_0^{T} \E |g(\xi_s(x))|^{p(1+\delta/2)}ds+
\sup_{\varepsilon <\varepsilon_0}\sup_{x \in S}\int_0^{T} \E |g(\xi_s^{\varepsilon}(x))|^{p(1+\delta/2)}ds
<\infty.
\end{equation}

To estimate the second factor note that the vector fields $A_l$, and $A_l^\varepsilon$ are bounded uniformly in $\varepsilon$ for sufficiently small $\varepsilon$ and so by standard estimates for bounded set $S$ in $\R^n$, and $T\in (0, T_1]$, 
\begin{equation}\label{N13}
\sup_{\varepsilon<\varepsilon_0} \sup_{x \in S} \E\sup_{0\leqslant s \leqslant T}|\xi_s^{\varepsilon}(x)|^2+
 \sup_{x \in S} \E\sup_{0\leqslant s \leqslant T}|\xi_s(x)|^2 <\infty.
\end{equation}

Hence there exists a constant $C(K,\alpha, M,\theta, n, T_0, \tilde{T}, S)$, not depending  on $R$ or $\varepsilon$, such that, 
\begin{equation*}%\label{N03}
\sup_{\varepsilon<\varepsilon_0}
\sup_{x \in S}I^{\varepsilon}_{12}(x,T,R)\leqslant \frac{C}{R^{\frac{2\delta}{2+\delta}}}
\end{equation*}

Finally we obtain, for each $R>0$, $\varepsilon<\varepsilon_0$,
$$\sup_{x \in S}I_1^{\varepsilon}(x,T_1 )\leqslant \sup_{x \in S}
I_{11}^{\varepsilon}(x,T_1,R)+ \sup_{x \in S}I_{12}^{\varepsilon}(x,T_1,R)\leqslant \sup_{x \in S} I_{11}^{\varepsilon}(x,T_1,R)+\frac{C}{R^{\frac{2\delta}{2+\delta}}}.$$
  First let $\varepsilon$ tend to 0, taking into account of (\ref{N02}) , then $R$ tend to $\infty$, we see
$$\lim_{\varepsilon \rightarrow 0}\sup_{x \in S} I_1^{\varepsilon}(x,T_1)=0.$$
\vskip 3pt

Next we observe that locally integrable functions are uniformly continuous on sufficiently large  sets.
 By  Egoroff  theorem for finite measures if $f_n$ converges almost surely, for any $\zeta>0$ it converges  uniformly outside of a  set of measure $\zeta$. So for any $g \in \lloc(\R^n)$, there is a function $\widetilde{g}$, such that 
$g=\widetilde{g}$ almost everywhere with Lebesgue measure in $\R^n$, and for positive numbers  $R$ and $\zeta$,  
there exists a open set $U(R,\zeta)$ with the property that
$\lambda(U(R,\zeta))<\zeta$ and    $\widetilde{g}$ is uniformly continuous on $\bar B_R/U$. So for any $r>0$, 
there exists a $\vartheta \equiv \vartheta(R,\zeta,r)$ be such that
$$|\widetilde{g}(x_1)-\widetilde{g}(x_2)|<r, \qquad  \forall x_i\in  \overline{B_{R}}\setminus U(R,\zeta), \;  |x_1-x_2|<\vartheta(R,\zeta,r).$$

Note that by Lemma \ref{Lem N2}, given the function $g=\widetilde{g}$ almost everywhere with Lebesgue measure, we 
have $g(\xi_s(x,\omega))=\widetilde{g}(\xi_s(x,\omega))$ and 
$g(\xi^{\varepsilon}_s(x,\omega))=\widetilde{g}(\xi^{\varepsilon}_s(x,\omega))$ almost surely in the probability space
for each fixed $s>0$ and $0<\varepsilon<\varepsilon_0$. Then for each $0\leqslant s \leqslant T_1$, let
\begin{eqnarray*}
O_1(s)&=&\{\omega:\ |\xi_s^{\varepsilon}(x,\omega)-\xi_s(x, \omega)|<\vartheta(R,\zeta,r)| \},\\
O_2(s)&=&\{\omega:\ \xi_s(x,\omega)\in \overline{B_{R}}\setminus U(R,\zeta)\}
\cap \{\omega: \xi_s^{\varepsilon}(x,\omega)
\in \overline{B_{R}}\setminus U(R,\zeta)\}.
\end{eqnarray*}
For each $0<T\leqslant T_1$, we obtain, 
\begin{equation*}
\begin{split}
&I^{\varepsilon}_2(x,T)=C\int_0^T
\E|\widetilde{g}(\xi_s^{\varepsilon}(x))-\widetilde{g}(\xi_s(x))|^p ds\\
&\leqslant C\int_0^T \E\big[|\widetilde{g}(\xi_s^{\varepsilon}(x))-g(\xi_s(x))|^p
\1_{\{O_1(s)\cap O_2(s)\}}(\omega)\big]ds\\
&+C\int_0^T \E\big[|\widetilde{g}(\xi_s^{\varepsilon}(x))-\widetilde{g}(\xi_s(x))|^p
\1_{\{O_1(s)\cap O_2(s)\}^c}(\omega)\big]ds\\
& \leqslant C r^p+C \Big(\int_0^T 
\E[|\widetilde{g}(\xi_s^{\varepsilon}(x))|^{p(1+\delta/2)}+
|\widetilde{g}(\xi_s(x))|^{p(1+\delta/2)}]ds\Big)^{\frac{2}{2+\delta}}
\Big(\int_0^T 
\p(O_1^c(s)\cup O_2^c(s))ds\Big)^{\frac{\delta}{2+\delta}}
\end{split}
\end{equation*}
By the estimate (\ref{G12}), \begin{equation*}
\begin{split}
&\sup_{x \in S}I^{\varepsilon}_2(x,T)\leqslant C r^p 
+C\sup_{x \in S}\Big(\int_0^{T} 
\big(\p\left(|\xi_s^{\varepsilon}(x)-\xi_s(x)|>\vartheta(R,\zeta,r)\right)\\
&+\p(|\xi_s^{\varepsilon}(x)|>R)
+\p(|\xi_s(x)|>R)
+\p(\xi_s^{\varepsilon}(x) \in U(R,\zeta))+
\p(\xi_s(x) \in U(R,\zeta))\big)ds\Big)^{\frac{\delta}{2+\delta}}\\
&\leqslant Cr^p +C\sup_{x \in S}\Big(\frac{\E\sup_{0\leqslant s \leqslant T}|\xi_s^{\varepsilon}(x)-\xi_s(x)|^2}
{\vartheta(R,\zeta,r)^2}
+\frac{\E[\sup_{0\leqslant s \leqslant T}\big(|\xi_s^{\varepsilon}(x)|^2+|\xi_s(x)|^2\big)]}{R^2}\\
&+\int_{0}^{T_3}\big(\p\left(\xi_s^{\varepsilon}(x) \in U(R,\zeta)\right)+
\p\left(\xi_s(x) \in U(R,\zeta)\right)\big)ds\Big)^{\frac{\delta}{2+\delta}}
\end{split}
\end{equation*}
Then by Lemma \ref{N2} and estimate (\ref{N13}), let $\varepsilon$ tend to $0$, we have,
\begin{equation}\label{N14}
\begin{split}
&\varlimsup_{\varepsilon \rightarrow 0}\sup_{x \in S}I^{\varepsilon}_2(x,T)\leqslant C r^p\\
&+C\Big(
\frac{1}{R^2}
+\int_{0}^{T}\sup_{x \in S}\big(\p\left(\xi_s^{\varepsilon}(x) \in U(R,\zeta)\right)+
\p\left(\xi_s(x) \in U(R,\zeta)\right)\big)ds\Big)^{\frac{\delta}{2+\delta}}
\end{split}  
\end{equation}

The last two items  can be estimated using the Markov kernel upper bounds and H\"older inequality as below,
\begin{equation*}
\begin{cases}
&\sup_{x \in S}\p\big(\xi_s^{\varepsilon}(x) \in U(R,\zeta)\big)\leqslant C s^{-\frac{1}{2}}
(\lambda(U(R,\zeta)))^{\frac{1}{n}}\leqslant C \zeta^{\frac{1}{n}} s^{-\frac{1}{2}}\\
%\label{e12}
& \sup_{x \in S}\p\big(\xi_s(x) \in U(R,\zeta)\big)\leqslant C \zeta^{\frac{1}{n}} s^{-\frac{1}{2}}
 \end{cases}
\end{equation*}

Put the above estimate into (\ref{N14}), we derive
\begin{equation*}
\begin{split}
\varlimsup_{\varepsilon \rightarrow 0}\sup_{x \in S}I^{\varepsilon}_2(x,T_1)
\leqslant C r^p +\Big( \frac{1}{R^2}+ \zeta^{\frac{1}{n}}\Big)^{\frac{\delta}{2+\delta}}
\end{split}
\end{equation*}

Since $R$, $\zeta$ and $r$ are arbitrary, then let $\zeta, r$ tend to $0$ and
$R$ tend to $\infty$, then we get $\lim_{\varepsilon \rightarrow 0}\sup_{x \in S}I^{\varepsilon}_2(x,T_1)=0$ and by now
we have completed the proof.
\end{proof}

\begin{remark}
Let $n>1$ and $g:\R^n\to \R$ be a function such that $g \in L^{\overline{p}(n)}_{\loc}(\R^n) $,
as (\ref{N6}) and the following, for some $R_0>0$ and $c>0$,
\begin{equation}\label{N15}
 |g(x)|\leqslant \e^{c|x|^2}, \qquad  |x|\geqslant R_0.
\end{equation}
Then  condition (\ref{N7}) holds for $T<\frac{1}{4c}$. In fact
\begin{equation*}
\begin{split}
&\sup_{x \in S}\int_0^{T}\int_{|y|\leqslant R_0}|g(y)|^{p(1+\delta/2)}K_s(x,y)dyds\\
&\leqslant C\int_0^{T}s^{-\frac{2+\delta}{2+2\delta}}
\big(\int_{|y|\leqslant R_0} |g|^{\frac{pn(1+\delta)}{2}}(y)dy \big)^{\frac{2+\delta}{n(1+\delta)}}ds<\infty
\end{split}
\end{equation*}
and, by a change of variable,
\begin{align*}
\sup_{x \in S}\int_0^{T}  \int_{|y|\geqslant R_0}& |g(y)|^{p(1+\delta/2)}  K_s(x,y) \;dy\,ds\\
&\leqslant \sup_{x \in S}\int_0^{T}s^{-\frac{n}{2}}\int_{\R^n}\e^{c\,|y|^2}
\e^{-\frac{|x-y|^2}{2s}}dyds\\
& \leqslant C\sup_{x \in S}\int_0^{T} \int_{\R^n}\e^{2c\,(s|y|^2+|x|^2)}
\e^{-\frac{|y|^2}{2}}dyds\\
&\leqslant  CT\sup_{x \in S}\e^{2c\,|x|^2}\int_{\R^n}\e^{2c\,T|y|^2}
\e^{-\frac{|y|^2}{2}}dy<\infty.
\end{align*}
\end{remark}

%\begin{definition}
%Let  $G$ be a function with the form $\{A_i\}_{i=0}^l$ define $G^\varepsilon$ to be the corresponding function with each $A_i$ replaced by $A_i^\varepsilon$. Let  $\sigma, T_0$ be positive constants.
%We say that condition $G(\sigma, T_0)$ holds  for  $G$,   if for any bounded set $S$ in $\R^n$
%\begin{equation}\label{N17}
%\sup_{x \in S}\int_0^{T_0}\int_{\R^n}\e^{\sigma G(y)}K_s(x,y)dyds<\infty. 
%\end{equation}
%\end{definition}

\begin{lemma}
\label{Lem N3}
%Take $T_0\in (0,\tilde T_1]$.
Let $G(x):=\sum_{l=0}^m |DA_l(x)|^2$ and $G^{\varepsilon}(x):=\sum_{l=0}^m |DA_l^{\varepsilon}(x)|^2$. Assume Assumption \ref{N1} and that there exist positive constants $\sigma$ and $T_0$, such that for any bounded set $S$ in $\R^n$, 
%following condition $G(\sigma, T_0)$ holds for $G$.
\begin{equation}\label{N17}
\sup_{x \in S}\int_0^{T_0}\int_{\R^n}\e^{\sigma G(y)}K_s(x,y)dyds<\infty. 
\end{equation}
Then for each $q>0$, 
 \begin{equation*}
\sup_{\varepsilon<\varepsilon_0}\sup_{x \in S} \E \Big[\e^{6q^2\int_0^{T_2}G^{\varepsilon}
(\xi_s^{\varepsilon}(x))ds} \Big]<\infty
\end{equation*}
for $T_2:=\min(T_1, \frac{\sigma}{6q^2})$ where $T_1=\min(\tilde T,\frac{T_0} {C_1})$ for $C_1$  given by (\ref{N3}).
\end{lemma}

\begin{proof}
By Jesen's inequality, 
\begin{eqnarray*}
\sup_{x \in S} \E \Big[\e^{6q^2\int_0^{T}G^{\varepsilon}
(\xi_s^{\varepsilon}(x))ds} \Big]
&\leqslant& \frac{1}{T} \sup_{x \in S} \E \Big[\int_0^T\e^{6Tq^2G^{\varepsilon}
(\xi_s^{\varepsilon}(x))}ds \Big]\\
&\leqslant&  \frac{1}{T} \sup_{x \in S} 
\E \Big[\int_0^T \big(\eta_{\varepsilon}\ast \e^{6Tq^2G}\big)
(\xi_s^{\varepsilon}(x))ds \Big]
\end{eqnarray*}
%Since,  (\ref{N17}),  $\e^{\sigma G(x)} \in L_{\loc}^{1}(\R^n)$ we may apply the previous lemma. 
the function $\e^{6Tq^2G}$ satisfies (\ref{N7-1}) with power parameter $p=1$, 
when $T\leqslant \frac{\sigma} {6q^2}$, then by Lemma \ref{lemma3-1} for 
$ T_2:=\min(T_1,\frac{\sigma} {6q^2})$, 
\begin{equation*}
\sup_{\varepsilon<\varepsilon_0}\sup_{x \in S} 
\E \Big[\int_0^{T_2} \big(\eta_{\varepsilon}\ast \e^{6Tq^2G}\big)
(\xi_s^{\varepsilon}(x))ds \Big]
\leqslant  \sup_{\varepsilon<\varepsilon_0}\sup_{x \in S} 
\E \Big[\int_0^{T_1} \big(\eta_{\varepsilon}\ast \e^{\sigma G}\big)
(\xi_s^{\varepsilon}(x))ds \Big]<\infty. 
\end{equation*}
which implies the conclusion of the lemma.
\end{proof}

\begin{remark}
In fact, in integrable condition (\ref{N17}), if we replace the function
$\e^{\sigma G(y)}$ by $F(G(y))$, where $F:\R^{+}\rightarrow \R^{+}$ is a convex non-negative function, then a corresponding uniformly integrability conclusion  holds for $F(G)$. 
\end{remark}

\subsection{Convergence of derivative flows}

Next we  consider the convergence of the derivative flows.  Since each $A_l^{\varepsilon}$ is smooth and  globally Lipschitz continuous, for each $\varepsilon$ there  is a smooth global solution flow $\xi_t^{\varepsilon}(x,\omega)$ to  SDE (\ref{e2}).
 For $x \in \R^n$, let
$$V_t^{\varepsilon}(x)=D_x \xi_t^{\varepsilon}$$
be the space derivative of $\xi_t^{\varepsilon}$. Then $V_t^{\varepsilon}$ satisfies
the following SDE
\begin{equation}\label{e3}
V_t^{\varepsilon}(x)= \mathbf{I}+\sum_{l=1}^{m}\int_0^t DA_l^{\varepsilon}(\xi_s^{\varepsilon}(x))
(V_s^{\varepsilon}(x))dW_s^l+ \int_0^tDA_0^{\varepsilon}(\xi_s^{\varepsilon}(x))(V_s^{\varepsilon}(x))ds.
\end{equation}
We  prove the following uniform moment estimate for $V_t^{\varepsilon}$.
\begin{lemma}\label{Lem N4}
%Suppose that Assumption \ref{N1} holds and there exist $\sigma >0$ and $T_0>0$, the 
%conditions (\ref{N17}) are satisfied, 
We assume the same condition as that in Lemma \ref{Lem N3} . Then for each $p>0$, 
there is a constant $T_3:=\min(T_1,{\sigma \over 6p^2})$, such that for all 
$0\leqslant T \leqslant T_3$ and bounded subset $S$ of  $\R^n$,   
\begin{equation}\label{N19}
\sup_{\varepsilon<\varepsilon_0}\sup_{x \in S} \E \Big[\sup_{0\leqslant s \leqslant T}|V_s^{\varepsilon}(x)|^p \Big]<\infty.
\end{equation}
\end{lemma}

\begin{proof} 
For simplicity we omit the starting point $x$ in $V_t^\varepsilon(x)$. For  SDE (\ref{e2}), with smooth coefficients, it holds that for all $p>1$,  see the analysis in \cite{Li1},
 \begin{equation}\label{N18}
|V_t^{\varepsilon}|^p=\e^{M_t^{p,\varepsilon}-\frac{\langle
M^{p,\varepsilon},M^{p,\varepsilon}\rangle_t}{2}+a_t^{p,\varepsilon}} 
\end{equation}
where 
\begin{equation*}
M_t^{p,\varepsilon}=\sum_{l=1}^{m}p\int_0^t
\frac{\langle DA_l^{\varepsilon}(V_s^{\varepsilon}),V_s^{\varepsilon}\rangle}{|V_s^{\varepsilon}|^2} 
dW_s^l,\ a_t^{p,\varepsilon}=\frac{p}{2}\int_0^t
\frac{H_p^{\varepsilon}(\xi_s^{\varepsilon})(V_s^{\varepsilon},V_s^{\varepsilon})}{|V_s^{\varepsilon}|^2}
\end{equation*}
and 
\begin{equation}\label{N16}
H_p^{\varepsilon}(v,v)=2\langle DA_0^{\varepsilon}(v), v\rangle
+\sum_{l=1}^{m}|DA_l^{\varepsilon}(v)|^2+(p-2)\sum_{l=1}^{m} 
\frac{\langle DA_l^{\varepsilon}(v),v\rangle^2}{|v|^2}.
\end{equation}
This follows from an It\^o formula applied to the function $|-|^p$ and  to the stochastic process $V_t^\varepsilon$. See Elworthy's book \cite{Elworthy} for a nice It\^o formula. 

Let $G^{\varepsilon}(x)=\sum_{l=0}^{m}|DA_l^{\varepsilon}(x)|^2$. Note that
$$H^{\varepsilon}_p(x)(v,v)\leqslant (p+3) G^{\varepsilon}(x)+C.$$ By Lemma 
\ref{Lem N3}, for $T_3:=\min(T_1,{\sigma \over 6p^2})$, 
%fro $T\le\tilde T_0(M, \theta, n, T_0, p, \sigma)$,
\begin{equation*}
\sup_{\varepsilon<\varepsilon_0}\sup_{x \in S} \E \Big[\e^{6p^2\int_0^{T_3}G^{\varepsilon}
(\xi_s^{\varepsilon}(x))ds} \Big]<\infty  
\end{equation*}
According to the proof of Theorem 5.1 in \cite{Li1}, for any bounded set $S$ in $\R^n$, 
\begin{equation*}
\sup_{\varepsilon<\varepsilon_0}\sup_{x \in S} \E \Big[\sup_{0\leqslant s \leqslant T_3}|V_s^{\varepsilon}(x)|^p \Big]
\leqslant C \sup_{\varepsilon<\varepsilon_0} \sup_{x \in S} \E \Big[\e^{6p^2\int_0^{T_3}G(\xi_s^{\varepsilon}(x))ds} \Big]<\infty 
\end{equation*}
Here $C$ is a constant only depending on $n$ and $p$, not on $\varepsilon$.
\end{proof}

\begin{remark}
\label{remark-G}
The  condition (\ref{N17}) used in the Lemma is a little stronger than it is needed for the uniformly moment estimate for $V_s^{\varepsilon}(x)$.
 In fact, let $F_1(x):=\sup_{|v|=1}\langle DA_0^{\varepsilon}(v), v\rangle(x)$
and  $F_2(x):=\sum_{l=1}^m |DA_l(x)|^2$, we have for any $\varepsilon>0$, 
\begin{equation*}
\begin{split}
&\sup_{|v|=1}\langle DA_0^{\varepsilon}(x)(v), v\rangle
=\sup_{|v|=1}\langle \int_{\R^n}\eta_{\varepsilon}(x-y)DA_0(y)(v)dy, v\rangle\\
&\leqslant \int_{\R^n}\eta_{\varepsilon}(x-y)\sup_{|v|=1}\langle DA_0(y)(v)dy, v\rangle
\leqslant \eta_{\varepsilon}*F_1(x)
\end{split}
\end{equation*}
So by Jensen's inequality, from (\ref{N16}) we see,
\begin{equation*}
\sup_{|v|=1}H_p^{\varepsilon}(v,v) \leqslant \eta_{\varepsilon}*\big( F_1(x)+F_2(x)\big)
\end{equation*}
So by (\ref{N18}), H\"older inequality and Jesen's inequality, if the following condition holds, 
%condition $G(\sigma, T_0)$ holds  for $G=F_i$, i.e.
\begin{equation}\label{l17}
\sup_{x \in S}\int_0^{T_0}\int_{\R^n}\e^{\sigma F_i(y)}K_s(x,y)dyds<\infty, \ \ i=1,2 
\end{equation}
 we  obtain the uniformly moment estimate for $V_s^{\varepsilon}(x)$ at some small time interval.

Note that in condition (\ref{l17}), we only need the one-side bound of $DA_0$, which is weaker than the two-side bound condition (\ref{N17}).
\end{remark}

\begin{thm}\label{th-derivative}
Suppose the Assumption \ref{N1}  and condition 
%$G(\sigma, T_0)$ holds, 
(\ref{N17}) holds,
%for $G=F_i$ in Remark \ref{remark-G}.  
Then for each $p>0$ , there is a constant $T_4$, such that for
any bounded set $S$ in $\R^n$ and $0\leqslant T \leqslant T_4$   
\begin{equation}\label{N20}
\lim_{\varepsilon, \tilde{\varepsilon}\rightarrow 0}\sup_{x \in S}\E
\sup_{0\leqslant s \leqslant T }|V_s^{\varepsilon}(x)-V_s^{\tilde{\varepsilon}}(x)|^p=0.
\end{equation}
\end{thm}

%Before we prove Theorem \ref{th-derivative}, 

\begin{proof}
%We may take a version of  $DA_l\in \loc(\R^n)$,  vector fields in the SDE (\ref{e4}),  as constructed in the proof of Lemma \ref{lemma5}
%so that  they are continuous locally on large sets.
We  only need to consider the case of $p\geqslant 2$. For simplicity, we use $\beta_l^{\varepsilon}(s)$, $\beta_l^{\tilde{\varepsilon}}(s)$ 
to denote $DA_l^{\varepsilon}(\xi_s^{\varepsilon}(x))$ 
and $DA_l^{\tilde{\varepsilon}}(\xi_s^{\tilde{\varepsilon}}(x))$, and 
the constants $C$ may appears in the computation 
from line to line and depend only on $K,\alpha, M,\theta,\delta,n,\tilde{T},p,\sigma$.
 Let $\hat{T}_1$ be the constant  $T_3$  in Lemma \ref{Lem N4} for the power parameter $2p$.
By  SDE (\ref{e4}), for any $T<\hat{T}_1$, we have, 
\begin{equation*}%\label{N21}
\begin{split}	
&\Big(\E\sup_{0\leqslant s \leqslant T}\left|V^{\varepsilon}_s(x)-V_s^{\tilde{\varepsilon}}(x)
\right|^p\Big)^{\frac{2}{p}}\\
&\leqslant C\sum_{l=1}^m \Big(\E\sup_{0\leqslant s \leqslant T} \left |   \int_0^s 
\beta_l^{\varepsilon}(u)
(V_u^{\varepsilon}(x))-\beta_l^{\tilde{\varepsilon}}(u)
(V_u^{\tilde{\varepsilon}}(x)) dB_u^l \right |^p\Big)^{\frac{2}{p}}\\
&+C \Big(\E\sup_{0\leqslant s \leqslant T}\left|\int_0^s 
\beta_0^{\varepsilon}(u)
(V_u^{\varepsilon}(x))-\beta_0^{\tilde{\varepsilon}}(u)
(V_u^{\tilde{\varepsilon}}(x))ds \right|^p\Big)^{\frac{2}{p}}\\
&\leqslant C\sum_{l=0}^m \Big(\E\Big[\int_0^T |\beta_l^{\varepsilon}(s)
(V_s^{\varepsilon}(x))-\beta_l^{\tilde{\varepsilon}}(s)
(V_s^{\tilde{\varepsilon}}(x))|^2 ds\Big]^{\frac{p}{2}}\Big)^{\frac{2}{p}}\\
&\leqslant C\sum_{l=0}^m \int_0^T \Big(\E |\beta_l^{\varepsilon}(s)
(V_s^{\varepsilon}(x))-\beta_l^{\tilde{\varepsilon}}(s)
(V_s^{\tilde{\varepsilon}}(x))|^p\Big)^{\frac{2}{p}}ds,.
\end{split}\end{equation*}
where the second step of above inequality is due to 
BKG inequality and H\"older inequality, the third step 
is due to the inequality $\Big(\E|\int_0^T|f_s|ds|^p\Big)^{\frac{1}{p}}
\leqslant \int_0^T \big(\E|f_s|^p\big)^{\frac{1}{p}}ds$ for measurable function
$f(s,\omega)$ when $p\geqslant 1$. Now splitting up the terms,
\begin{equation}\label{N21}
\begin{split}	
&\Big(\E\sup_{0\leqslant s \leqslant T}\left|V^{\varepsilon}_s(x)-V_s^{\tilde{\varepsilon}}(x)
\right|^p\Big)^{\frac{2}{p}}\\
&\leqslant C\sum_{l=0}^m \int_0^T \Big(\E
|\beta_l^{\varepsilon}(s)
(V_s^{\varepsilon}(x))-\beta_l^{\varepsilon}(s)
(V_s^{\tilde{\varepsilon}}(x))|^p \Big)^{\frac{2}{p}}ds\\
&+ C\sum_{l=0}^m  \int_0^T
\Big(\E|\beta_l^{\varepsilon}(s)
(V_s^{\tilde{\varepsilon}}(x))-\beta_l^{\tilde{\varepsilon}}(s)
(V_s^{\tilde{\varepsilon}}(x))|^p\Big)^{\frac{2}{p}} ds\\
&\leqslant CN^2 \int_0^T \Big(\E\sup_{0\leqslant u \leqslant s}
|V_u^{\varepsilon}(x)-V_u^{\tilde{\varepsilon}}(x)|^p\Big)^{\frac{2}{p}} ds\\
&+ 
C\sup_{\varepsilon< \varepsilon_0}
\Big(\E\Big[\sup_{0\leqslant s \leqslant \hat T_1}|V_s^{\varepsilon}|^{2p}\Big]\Big)^{\frac{1}{p}}
\Big(\sum_{l=0}^m \big(\int_0^T\E|\beta^{\varepsilon}_l(s)|^{4p}ds\big)^{\frac{1}{2p}}
\big(\int_0^T \p(|\beta_l^{\varepsilon}(s)|>N)ds\big)^{\frac{1}{2p}}\Big)\\
&+ C\sup_{\varepsilon< \varepsilon_0}
\Big(\E\Big[\sup_{0\leqslant s \leqslant \hat T_1}|V_s^{\varepsilon}|^{2p}\Big]\Big)^{\frac{1}{p}}\Big(\sum_{l=0}^m  
\big(\int_0^T
\E|\beta_l^{\varepsilon}(s)
-\beta_l^{\tilde{\varepsilon}}(s)|^{2p}
ds\big)^{\frac{1}{p}}\Big).
\end{split}	
\end{equation}
The last step is due to 
H\"older's inequality. Note that by condition  
(\ref{N17}), we can find a constant $L>0$ (depends on $M,\theta,n,\tilde{T},\sigma$), such that
the function $g=\e^{L|DA_l|^2}$ satisfies the condition (\ref{N7-1}) with 
power parameter $p=1$. And then $|DA_l|$ satisfies the conditions (\ref{N6}) and (\ref{N7}) 
for any power $p\geqslant 1$, 
%$|\beta^{\varepsilon}_l(s)|^{4p}\leqslant C\e^{L_2|\beta^{\varepsilon}_l(s)|^2}$, 
so by Lemma \ref{lemma3-1} and Lemma \ref{lemma3}, there are constants $\hat{T}_2$ depending on
$K,\alpha, M,\theta,n,\tilde{T},T_{0},\sigma$, such that for any bounded set $S$, 
\begin{equation*}
\begin{split}
\sup_{\varepsilon<\varepsilon_0}\sup_{x \in S}\int_0^{\hat{T}_2}
\E \e^{L|\beta^{\varepsilon}_l(s)|^2}ds \leqslant
\sup_{\varepsilon<\varepsilon_0}\sup_{x \in S}\int_0^{\hat{T}_2}
\E \Big[\eta_{\varepsilon}*\e^{L|DA_l|^2} (\xi_s^{\varepsilon}(x))\Big]ds< 
\infty
\end{split} 
\end{equation*}

\begin{equation}\label{N23}
\sup_{\varepsilon<\varepsilon_0}\sup_{x \in S}\int_0^{\hat{T}_2} \E
|\beta^{\varepsilon}_l(s)|^{4p} ds< \infty
\end{equation}
\begin{equation}\label{N24}
\lim_{\varepsilon, \tilde{\varepsilon}\rightarrow 0}\sup_{x \in S}\int_0^{\hat{T}_2}
\E|\beta_l^{\varepsilon}(s)
-\beta_l^{\tilde{\varepsilon}}(s)|^{2p}ds=0
\end{equation}

Then by the Chebeshev inequality, for each $\varepsilon <\varepsilon_0$, 
\begin{equation}\label{N22}
\int_0^{\hat{T}_2} \p(|\beta_l^{\varepsilon}(s)|>N)ds\leqslant
\int_0^{\hat{T}_2} \frac{\E \e^{L|\beta^{\varepsilon}_l(s)|^2}}{\e^{LN^2}} ds
\leqslant \frac{C}{\e^{L N^2}}
\end{equation}
So put (\ref{N19}), (\ref{N22}) and (\ref{N23}) into (\ref{N21}),
when $T<\min(\hat{T}_1, \hat{T}_2)$ we derive,
\begin{equation*}
\begin{split}
&\Big(\E\sup_{0\leqslant s \leqslant T}\left|V^{\varepsilon}_s(x)-V_s^{\tilde{\varepsilon}}(x)
\right|^p\Big)^{\frac{2}{p}}
\leqslant CN^2 \int_0^T \Big(\E\sup_{0\leqslant u \leqslant s}
|V_u^{\varepsilon}(x)-V_u^{\tilde{\varepsilon}}(x)|^p\Big)^{\frac{2}{p}} ds\\
&+\frac{C}{\e^{LN^2/2p}}+C\Big(\sum_{l=0}^m  
\big(\int_0^T
\E|\beta_l^{\varepsilon}(s)
-\beta_l^{\tilde{\varepsilon}}(s)|^{2p}
ds\big)^{\frac{1}{p}}\Big)
\end{split}	
\end{equation*}

By Gronwall lemma, let 
$\alpha^{\varepsilon,\tilde{\varepsilon}}(T,x):= \sum_{l=0}^m  
\big(\int_0^T
\E|\beta_l^{\varepsilon}(s)
-\beta_l^{\tilde{\varepsilon}}(s)|^{2p}
ds\big)^{\frac{1}{p}}$, we have for any $T<\min(\hat{T}_1, \hat{T}_2)$, 
\begin{equation*}
\begin{split}
&\Big(\E\sup_{0\leqslant s \leqslant T}|V^{\varepsilon}_s(x)-V_s^{\tilde{\varepsilon}}(x)|^p\Big)^{\frac{2}{p}}\\
&\leqslant \frac{C}{\e^{LN^2/2p}}+ \alpha^{\varepsilon,\tilde{\varepsilon}}(T,x)+\int_0^T \e^{CN^2(t-s)}
\Big(\frac{C}{\e^{LN^2/2p}}+\alpha^{\varepsilon,\tilde{\varepsilon}}(s,x)\Big)ds\\
&\leqslant  \alpha^{\varepsilon,\tilde{\varepsilon}}(T,x)+\e^{CN^2T}\int_0^T \Big(\frac{C}{\e^{LN^2/2p}}+\alpha^{\varepsilon,\tilde{\varepsilon}}(s,x)\Big)ds
\end{split}
\end{equation*}
Since by (\ref{N24}), $\lim_{\varepsilon,\tilde{\varepsilon} \rightarrow 0}\sup_{x \in S}\alpha^{\varepsilon,\tilde{\varepsilon}}(T,x)=0$,
first let $\varepsilon$ tend to $0$, then $N$ tend to infinity in above inequality, so we can find a 
constant $T_4>0$(take $CT_4<\frac{L}{2p}$ in above inequality), such that for any $T\leqslant T_4$, 
\begin{equation*}
\lim_{\varepsilon,\tilde{\varepsilon} \rightarrow 0}\sup_{x \in S}\E\sup_{0\leqslant s \leqslant T}|V^{\varepsilon}_s(x)-V_s^{\tilde{\varepsilon}}(x)|^p 
=0
\end{equation*}

\end{proof}

By Theorem \ref{th-derivative}, we know $V_t^{\varepsilon}$ is a Cauchy sequence in the 
space $L^p(\p)$, so there is a limit process, in fact we have the following.

\begin{thm}
\label{Th2}
%Assume Assumption \ref{N1}  and $G(\sigma, T_0)$ holds for the function $G$ the functions $F_i$ from Remark \ref{remark-G}. 

Suppose  Assumption \ref{N1}  and condition 
%$G(\sigma, T_0)$ holds, 
(\ref{N17}) hold, then there exists a process $V_t(x), 0\leqslant t < \infty$, such that 
%for $G=F_i$ in Remark \ref{remark-G}.  
for each  $p>0$, there is a $T_4>0$ as in Lemma \ref{Lem N3}, 
%for each $x \in \R^n$ and $t\leqslant \tilde T_0$, there is a process 
%$V_t(x,\varepsilon)$, such that for any bounded set $S$ in $\R^n$
\begin{equation*}
\lim_{\varepsilon\rightarrow 0}\sup_{x \in S}
\E\sup_{0\leqslant s \leqslant T_4}|V_s^{\varepsilon}(x)-V_s(x)|^p=0 
\end{equation*}
hols for any bounded set $S$ in $\R^n$. 
%And the process $V_t(x)$ is the unique strong solution of SDE (\ref{e4}) in time interval
%$0\leqslant t \leqslant \tilde T_0$. 
Furthermore, the process $V_t(x)$ is the unique strong solution
of SDE (\ref{e4}) for all $t$.
%in time interval $0\leqslant t \leqslant T$ for any $T>0$.
\end{thm}

\begin{proof}  We write  $\xi_t^0:=\xi_t$  and $V_t^0:=V_t$.
By Lemma \ref{lemma3} and Theorem \ref{th-derivative}, it is shown that
the limit process $V_t(x)$ is the solution of SDE (\ref{e4}) in some time interval
$0\leqslant t \leqslant  T_4$. This gives the moment estimate  in Lemma \ref{Lem N4} for the case that $\epsilon=0$ on $[0,T_4]$.  In fact a direct computation as that in Lemma \ref{Lem N4}
gives the bound for any strong solution of SDE (\ref{e4}) on $[0,T_4]$. The key observation is that  equation (\ref{N18}) holds for any  strong solution of SDE (\ref{e4}), without further assumptions on the regularity on the vector fields. Hence if $V_t$ and $\tilde{V}_t$ are two solutions of 
SDE (\ref{e4}), the same method used for the proof of Theorem \ref{th-derivative}
gives
\begin{equation*}
\sup_{x \in S}\E\sup_{0\leqslant s \leqslant T_4}|V_s(x)-\tilde{V}_s(x)|^p=0 
\end{equation*}
and $V_t(x)$ is the unique strong solution of SDE (\ref{e4}) in the time interval
$[0,T_4]$. 

If we view SDE (\ref{e1}) and (\ref{e4}) together as 
a system with solution ($\xi_t(x), V_t(x)$) valued in $\R^n\times \R^{n\times n}$.
Let $F_t(x,v_0,\omega):=(\xi_t(x), \langle V_t(x), v_0\rangle v_0)$ which is the solution of 
that system with initial point ($x, v_0$).  
%Take $T_4$ to be $T_4/2$ for simplicity.
 When $T_4<t\leqslant 2T_4$, let 
 $$F_t(x,v_0,\omega):=
F_{t-T_4}\big(\xi_{\tilde T}(x), V_{T_4}(x),\theta_{T_4}(\omega)\big).$$  Here 
 $\theta_{\tilde T_{0}}(\omega)_t=\omega_{t+T_4}-\omega_{T_4}$ is the shift operator.
By the Markov property and the pathwise uniqueness at time interval
$0\leqslant t \leqslant T_4$ for any initial point $(x,v_0) \in \R^n\times \R^{n\times n}$, one may check that $V_t(x)$ is the solution for 
SDE (\ref{e4}) when $0\leqslant t \leqslant 2T_4$. %taking $\tilde T=T_4/2$. 
Taking this procedure repeatedly, we obtain a unique
solution to SDE (\ref{e4}) for any time $t$.
\end{proof}
\begin{remark}
In particular, by Lemma \ref{Lem N2}, if we take different versions of weak derivative $DA_l$ in SDE (\ref{e4}), the corresponding solutions $V_s$ are indistinguishable.
\end{remark}
\begin{remark}
In Theorem \ref{Th2},  $V_t(x)$ is shown to belong to $ \in L^p(\p)$ when 
$0\leqslant t \leqslant T_4$. But this may fail when $t>T_4$.   
\end{remark}

\section{The case of locally Lipschitz continuous coefficients} \label{lip}
In a special case that $A_l, 0\leqslant l \leqslant m$ are bounded, global Lipschitz continuous and uniformly
elliptic in 
$\R^n$, the condition (\ref{N17}) are satisfied for every $T_0>0, \sigma>0$. And
for each $T>0$, there exists a unique strong solution of SDE (\ref{e4}) since $DA_l$ is bounded. Step by step
checking the proof of Lemma \ref{lemma3} and Thereom \ref{th-derivative} to determine the time interval, we can 
obtain,
%for the term  $$\sup_{x\in S}I_1^\epsilon=\sup_{x\in S} \int_0^t\int_{\R^n} |g^\epsilon(y)-g(y)|^p\, K_s^\epsilon(x,dy)\;ds\to 0$$
% since the integrable is bounded by $K_s^\epsilon(x,y)\le C K_s(x,y)$. Also
%\begin{eqnarray*}
%\sup_{x\in S}I_2^\epsilon(x)  &=&C \sup_{x\in S} \E\int_0^T |\tilde g(\xi_s^\epsilon(x)-g(\xi_s(x))|^p ds\\
%&\le&C \sup_{x\in S} \E\int_0^T |\tilde g(\xi_s^\epsilon(x)-\tilde g(\xi_s(x))|^p\1_{O_1\cap O_2} ds\\
%&&+C \sup_{x\in S} \E\int_0^T |\tilde g(\xi_s^\epsilon(x)-\tilde g(\xi_s(x))|^p\1_{O_1^c\cap O_2^c} ds
%\\&\le& Cr^pT+2KT \, \sup_{x\in S}P(|\xi_s^\epsilon(x)-\xi_s(x)|\ge \vartheta) +
%2KT P(|\xi_s^\epsilon(x)|\ge R)
%\\&&+2KT P(|\xi_s(x)|\ge R)+P(\xi_s^\epsilon(x)\in U(R,\delta))
%+P(\xi_s(x)\in U(R,\delta))
%\end{eqnarray*}
%Take $\epsilon \to 0$ followed by $r\to 0$ and $\R\to \infty$:
%$$\sup_{x\in S}I_2^\epsilon(x)\le P\left(\xi_s(x)\in U(R,\delta)\right).$$
%Since $\xi_s(x)$ has a kernel with respect to the Lebesgue measure and
%take $\delta\to 0$ to see the concluding limit in Lemma \ref{lemma3}. The proof of  
%Theorem \ref{th-derivative}  does not rely on the assumption on the absolutely continuous of 
%the distribution of $\xi_s(x)$ and the conclusion  holds for in this case for all time $T$.

\begin{thm}\label{Th3}
Assume that the coefficients of the SDE  (\ref{e1}) are bounded, Lipschitz continuous and uniformly elliptic.
For each $T>0$, $p>0$ and bounded subset $S$ in $\R^n$, we have
$$\lim_{\varepsilon\rightarrow 0}\sup_{x \in S}\E
\sup_{0\leqslant s \leqslant T }|V_s^{\varepsilon}(x)-V_s(x)|^p=0.$$
\end{thm}

%\begin{remark} Note that
% $V_t$ is the mean derivative of $\xi_t$:
%$$\lim_{s\to 0} \E\left |{\xi_t(x+sv)-\xi_t(x)\over s}-\langle V_t^\varepsilon(x),v\rangle \right|^p=0.$$
%\end{remark}

We  extend the approximation results to the elliptic SDE with locally Lipschitz continuous coefficients, in which 
case $A_l$ is still weak differentiable and has a locally bounded version of the derivative. So SDE (\ref{e4}) is
complete if SDE (\ref{e1}) is complete, i.e. non-explode for each fixed starting point.

Denote by $(\rho,\vartheta),\ \rho>0,\ \vartheta \in \mathbb{S}^{n-1}$ the polar 
coordinate in $\mathbf{R}^n$. For any measurable function $f$ on $\mathbf{R}^n$ and integer $N>0$, 
define a function $f^N$ as,
\begin{equation}\label{h1}
f^N(\rho,\vartheta):=\begin{cases}\ f(\rho,\vartheta) &\text{if}\  |\rho|\leqslant N, \\
            f(N,\vartheta) &\text{if}\ |\rho|> N.
\end{cases}
\end{equation}
and $f^N(0)=f(0)$. Suppose that  SDE (\ref{e1}) is complete and  coefficients $A_l$ are locally Lipschitz continuous and elliptic . Set $A_l^N(x)=(A_{l1}^N(x),..,A_{ln}^N(x))$ then  $A_l,\ 1\leqslant l \leqslant m$ are bounded, Lipschitz continuous and
uniformly elliptic. 
%with a constant $\psi(N)$on $\R^n$, and
%for each $N$, $A_l^N(x)$ has the same Lipschitz constant
%with $A_l$. 
Let $\xi_t^N(x)$, $V_t^{N}(x)$ be respectively the solutions to 
SDE (\ref{e1}) and (\ref{e4}) whose coefficients are $A_l^N$ and $DA_l^N$.  Let
$A_l^{N,\varepsilon}$ be the smooth approximation of the vector fields $A_l^N$  derived 
by convolution, as in Section \ref{section2}. We denote the corresponding solution of approximation SDE
(\ref{e2}) and (\ref{e3}) by $\xi_t^{N,\varepsilon}(x)$ and $V_t^{N,\varepsilon}(x)$.

\begin{lemma}\label{Lem l0}
Suppose the coefficients $A_l$ of SDE (\ref{e1}) are locally Lipschitz continuous, 
and of linear growth, then for any $p>0,T>0$ and bounded set $S$ in $\R^n$,
\begin{equation}\label{l0}
\lim_{\varepsilon \rightarrow 0} \sup_{x \in S}
\E\sup_{0\leqslant s \leqslant T}|\xi_s^{\varepsilon}(x)-\xi_s(x)|^p=0, 
\end{equation}
%This is known: and  the distribution of the solution $\xi_s(x)$ of SDE (\ref{e1}) is absolutely continuous with respect to the Lebesgue measure in $\R^n$ for each fixed $s>0$ and  $x \in \R^n$  
\end{lemma}
\begin{proof}
Let $T^{\varepsilon}_N(x)$, $T_N(x)$ be the first exist time of the ball $B_N$ for the process
$\xi_s^{\varepsilon}(x)$, $\xi_s(x)$ respectively. Since $\xi_s^{N,\varepsilon}(x)=\xi_s^{\varepsilon}(x)$
a.s. for $s<T^{\varepsilon}_N(x)$,  and $\xi_s^{N}(x)=\xi_s(x)$ a.s for $s<T_N(x)$, we have,
\begin{equation}\label{l4}
\begin{split}
&\E\sup_{0\leqslant s \leqslant T}\Big[ |\xi_s^{N,\varepsilon}(x)-
\xi_s^{\varepsilon}(x)|^p+|\xi_s^{N}(x)-\xi_s(x)|^p\Big]\\
&\leqslant C\sup_{0<\varepsilon<\varepsilon_0,N>0}\sup_{x \in S}\Big(  \sqrt{\E\sup_{0\leqslant s \leqslant T} 
|\xi_s^{N,\varepsilon}(x)|^{2p}}\sqrt{\p(T>T^{\varepsilon}_N(x))}\\
&+ \sqrt{\E\sup_{0\leqslant s \leqslant T} 
|\xi_s^{N}(x)|^{2p}}\sqrt{\p(T>T_N(x))}\Big)\\
&\leqslant  \sup_{0<\varepsilon<\varepsilon_0,N>0}\frac{\sup_{x \in S}\left(   \E\sup_{0\leqslant s \leqslant T} 
|\xi_s^{N,\varepsilon}(x)|^{2p}+|\xi_s^{\varepsilon}(x)|^{2p}\right)}{N^p}.
\end{split}
\end{equation}
This convergence to $0$ as $N\to \infty$ from the uniform estimates below:
   \begin{equation}\label{l1}
\sup_{0<\varepsilon<\varepsilon_0,N>0}\sup_{x \in S}\left(   \E\sup_{0\leqslant s \leqslant T} 
|\xi_s^{N,\varepsilon}(x)|^p+|\xi_s^{\varepsilon}(x)|^p\right)<\infty. 
\end{equation}
The uniform estimate holds for any $p\geqslant 1$ and follows from the common linear bounded on $A_l^\epsilon$.
Since $A_l^N$ is bounded and global Lipschitz continuous for
each $N>0$, a Grownwall type argument shows that 
\begin{equation}\label{l2}
\lim_{\varepsilon \rightarrow 0} \sup_{x \in S}
\E\sup_{0\leqslant s \leqslant T}|\xi_s^{N,\varepsilon}(x)-\xi_s^N(x)|^p=0 
\end{equation}
By (\ref{l4}), (\ref{l2}) and (\ref{l1}), we conclude the proof by taking $\epsilon\to 0$ followed by $N\to \infty$ in the following inequality:
\begin{equation}\label{h3}
\begin{split}
 &\E\sup_{0\leqslant s \leqslant T} 
|\xi_s^{\varepsilon}(x)-\xi_s(x)|^p
\leqslant C  \E\sup_{0\leqslant s \leqslant T}
\left( |\xi_s^{N,\varepsilon}(x)-\xi_s^{\varepsilon}(x)|^p+|\xi_s^{N}(x)-\xi_s(x)|^p\right) \\
& +C \; \E\sup_{0\leqslant s \leqslant T} 
|\xi_s^{N,\varepsilon}(x)-\xi_s^{N}(x)|^p.
\end{split}
\end{equation}
%So put (\ref{l2}) and (\ref{l4}) into (\ref{l3}), first let $\varepsilon$ tend to $0$, then $N$ tend to $\infty$, we 
%can get (\ref{l0}).
%Conclusion is known.  Since the coefficients $A_l^N$ are bounded Lipschitz continuous and uniform elliptic for each $N>0$, 
%the distribution of  $\xi_s^N(x)$ is 
%absolutely continuous with respect to the Lebesgue measure in $\R^n$ for each fixed $s>0$ and 
%$x \in \R^n$. Note that we have proved (\ref{l0}), by the same approximation methods we adopted in
%the proof of Lemma \ref{Lem N2}, we can prove the distribution of  $\xi_s(x)$ is 
%absolutely continuous with respect to the Lebesgue measure in $\R^n$ for each fixed $s>0$ and 
%$x \in \R^n$ 
\end{proof}

\begin{remark}
If we assume coefficients $A_l$ of SDE (\ref{e1}) are locally Lipschitz continuous, elliptic 
and of linear growth. Since the coefficients $A_l^N$ are bounded Lipschitz continuous and uniform elliptic for each $N>0$, 
the distribution of  $\xi_s^N(x)$ is 
absolutely continuous with respect to the Lebesgue measure in $\R^n$ for each fixed $s>0$ and 
$x \in \R^n$. Note that we have proved (\ref{l0}), by the same approximation methods we adopted in
the proof of Lemma \ref{Lem N2}, we can prove the distribution of  $\xi_s(x)$ is 
absolutely continuous with respect to the Lebesgue measure in $\R^n$ for each fixed $s>0$ and 
$x \in \R^n$. In particular that if we take different versions of $DA_l$ in SDE (\ref{e4}), the solution $V_s$ are indistinguishable.   
\end{remark}

%The lemma above concluded,  
As the same argument in the 
proof Lemma \ref{Lem l0} above, especially triangle inequality (\ref{h3}) and the results of 
Theorem  \ref{Th3}, we present below an approximation lemma for $V_s$ in more general case and the remaining of the section devotes to the validity of the assumption there.

\begin{lemma}
\label{lemma5}
Let $S\subset \R^n$ be a bounded set and $T>0$. If 
\begin{equation}\label{g1}
\lim_{N\rightarrow \infty}\sup_{0<\varepsilon<\varepsilon_0}\sup_{x \in S} \E\sup_{0\leqslant t \leqslant T}
\left( |V_t^{N,\varepsilon}(x)-V_t^{\varepsilon}(x)|^p+|V_t^{N}(x)-V_t(x)|^p\right) =0 
\end{equation}
for all $p$, then
\begin{equation*}
\lim_{\varepsilon\rightarrow 0}\sup_{x \in S} \E\sup_{0\leqslant t \leqslant T} 
|V_t^{\varepsilon}(x)-V_t(x)|^p=0
\end{equation*}
\end{lemma}

%\begin{proof} 
%Note that 
%\begin{eqnarray*}
% \E\sup_{0\leqslant t \leqslant T} 
%|V_t^{\varepsilon}(x)-V_t(x)|^p
%&\le& C(p)  \E\sup_{0\leqslant t \leqslant T}
%\left( |V_t^{N,\varepsilon}(x)-V_t^{\varepsilon}(x)|^p+|V_t^{N}(x)-V_t(x)|^p\right) \\
%&&+C(p) \; \E\sup_{0\leqslant t \leqslant T} 
%|V_t^{N,\varepsilon}(x)-V_t^{N}(x)|^p
%\end{eqnarray*}
%for some constant $C(p)$. Apply Theorem \ref{th-derivative} to the SDE's with coefficients $\{A_l^N\}$, for each $N$,  to see 
%\begin{equation}\label{g3}
%\lim_{\varepsilon\rightarrow 0}\sup_{x \in S} \E\sup_{0\leqslant t \leqslant T} 
%|V_t^{N,\varepsilon}(x)-V_t^{N}(x)|^p=0.
%\end{equation}
%Take $\varepsilon\to 0$ followed by $N\to \infty$ for the required result.
%\end{proof}

 %The convergence, uniformly on compact time domain, of the derivative vector fields of the truncated SDE lead to bounds on the moments of the vector fields of the original SDEs, which is the extra control we needed for the convergence of the solution in an appropriate norm. 
The following theorem,  from Theorem 5.1 and observations from section 6 in  \cite{Li1},  are valid for 
strong solutions $\xi_t, V_t$ of SDE's (\ref{e1}) and  (\ref{e4}), which are not necessarily elliptic or with smooth 
coefficients. 
 %The diffusion vector fields  are assumed to be $C^3$, $C^2$ if the SDE is on $\R^n$ and written in It\^o form, and the drift vector field is $C^1$.  
\begin{thm}\label{old-thm}
\begin{itemize}
\item[(1)]  Suppose that there is a point $x_0$ such that the solution $\xi_t(x_0)$ exists for all time and
$$ \sup_{|v|=1} \langle DA_0(v), v\rangle (x) \leqslant f(x)|v|^2 \qquad \sum_{l=1}^m |DA_l|^2(x) \leqslant f(x)$$
for some function  $f: \R^n\to \R$. Then
 $$ \E \sup_{s\le t} |T\xi_t|^p <c\E \exp{\left( 6p^2\int_0^t f(\xi_s(x))ds\right)} $$
and the SDE is strongly complete.
\item   If $g:\R^n\to \R$ is a function such that $g \in C^2(\R^n)$
\begin{equation}\label{condition2}
{1\over 2}\sum_{l=1}^m |Dg(A_l)|^2+{1\over 2}\sum_{l=1}^mD^2g(A_l,A_l)+Dg(A_0)\leqslant  K
\end{equation}
for some constant $K$, then for all $\sigma$ and stopping times $\tau$, 
$$\E \exp{\left(\sigma g(\xi_{t\wedge \tau})(x)\right)}\leqslant \e^{\sigma g(x)+kt}$$ for some $k$ depending on $K$
and  the SDE is complete if $g$ has compact level sets.
\end{itemize}\end{thm}
The theorems and analysis we cited above from \cite{Li1}  are for the SDEs with smooth coefficients.
Our key observation is that when the  coefficients are not smooth,  (\ref{N18})) still holds for a strong solution. The same argument and technicalities applies and we obtain the conclusion above.

 The application of the theorem reduces to a well chosen function $f, g$ for a particular SDE.

\begin{asu}\label{J1}
Let $A_i:\R^n\to \R^n$, $l=0,1,\dots, m$, be locally Lipschitz continuous.
Let  $\psi_i:\R\to \R, i=1,2$ be positive non-decreasing functions and and let
$g_i(x):=\psi_i(|x|)$. Suppose that  $g_1$ is $C^2$ and 
%$g_2(x)\leqslant K_1(1+|x|)$,here $K_1$ is some positive constant. Besides that
the following holds: 
\begin{enumerate}
%\item[(0)] 
%Either $g_1$ has compact level sets or $\psi_1'(s)+\psi_1^{''}(s)+{(\psi_1'\psi)(s)\over 1+s}$ is bounded.
\item [(1)] $\sum_{l=1}^m|D A_l|^2(x) \leqslant  g_1(x)$, 
$\sup_{|v|=1}{\langle DA_0(x)(v),v\rangle }  \leqslant g_1(x)|v|^2$
\item[(2)]  $\sum_{l=0}^m |A_l|(x)\leqslant  g_2(x)\leqslant C_2(1+|x|)$, 
%$\langle A_0(x), x\rangle \le g_2(x) (1+|x|)$.
\item[(3)]
$${1\over 2}\sum_{l=1}^m |Dg_1(A_l)|^2+{1\over 2}\sum_{l=1}^m D^2g_1(A_l,A_l)+Dg_1(A_0)\leqslant C_3$$
\end{enumerate}
Here $C_2$ and $C_3$ are some  constants.
\end{asu}

Condition (3) in Assumption \ref{J1} is satisfied if $\psi_1 \in C^2(\R)$ and  
$\psi_1^{''}(s)(\psi_2(s))^2$ and $\psi^{'}_1(s)\psi_2(s)$ are bounded. It follows from the comments made earlier 
that  $\E e^{ \sigma g_1(\xi_{t}(x))}$ is finite for each $\sigma>0$.
%for any compact set $K$ and $\E \sup_{s\le t}|T\xi_s|^p<\infty$ for any $s$.
 \begin{remark}
Assumption \ref{J1} holds under one of the following conditions:
\begin{enumerate}
\item [(a)] $\sum_{l=1}^m|D A_l|$ is bounded.
\item[(b)] $\sum_{l=1}^m |DA_l(x)|^2 \leqslant C\big(1+\ln(1+|x|^2)\big), \qquad \sum_{l=0}^m | A_l(x)| \leqslant C(1+|x|)$, \\
$\langle x, A_0(x)\rangle\leqslant C(1+|x|^2)$, \qquad 
$\langle DA_0(x)(v), v\rangle \leqslant C\big(1+\ln(1+|x|^2)\big)|v|^2$.
\item[(c)]  For some $\delta\ge 0$, the following holds, \\
$\sum_{l=1}^m|A_l(x)|\leqslant C(1+|x|^2)^{{1\over 2}-\delta}$ and $\sum_{l=1}^m|DA_l(x)|^2\leqslant C(1+|x|^2)^{\delta}$ and \\
$\langle x, A_0(x)\rangle \leqslant C(1+|x|^2)^{1-\delta}$,
$\langle DA_0(x)(v), v\rangle \leqslant C(1+|x|^2)^{\delta}\,|v|^2$.
\end{enumerate} 
For part a) take $g_1$ to be a constant and $g_2$ a linear function . 
For (b) let $\psi_1(s)=\ln (1+s^2) $ and $\psi_2(s)=1+s$. For (c) Let $\psi_2(s)=C(1+s^2)^{{1\over 2}-\delta}$ and $\psi_1(s)=C(1+s^2)^{\delta}$. See Corollary 6.2 and 6.3 in \cite{Li1}.
\end{remark}

%\begin{lemma}
%Convegence lemma of $\xi_t^\varepsilon$ to $\xi_t$
%\end{lemma}
\begin{proposition}\label{J2}
Suppose that  $\Span \{A_1(x), \dots, A_m(x)\}=\R^n$ for each $x$ so (\ref{e1}) is elliptic. If in addition that  $\{A_0, A_1,\dots, A_m\}$ satisfies Assumption \ref{J1},
condition (\ref{g1}) in Lemma \ref{lemma5} holds. 
In particular, for each $T>0,p>0$ and bounded set $S$ in $\R^n$,
\begin{equation*}
\lim_{\varepsilon\rightarrow 0}\sup_{x \in S}\E
\sup_{0\leqslant s \leqslant T }|V_s^{\varepsilon}(x)-V_s(x)|^p=0. 
\end{equation*}
\end{proposition}

\begin{proof}
Here in the proof the constants $C$ may change in different lines and only depend on 
$p,S,T$. By Lemma \ref{lemma1}, for  all $N>0$ and $0<\varepsilon<\varepsilon_0$, 
$$|D A_l^{\varepsilon}(x)|^2+|D A_l^{N,\varepsilon}(x)|^2\leqslant 2 \psi_1(|x|+1), 
\qquad |A_l^{\varepsilon}(x)|+|A_l^{N,\varepsilon}(x)|\leqslant\psi_2(|x|+1).$$
So  there is a global solution to approximation SDEs with  smooth coefficients $A_l^{\varepsilon}$ and $A_l^{N,\varepsilon}$ for any starting point and  it follows from the assumption that  
$\tilde{g}(x):=\psi_1(|x|+1)$ is $C^2$ and the functions
\begin{eqnarray*}
&&{1\over 2}\sum_{l=1}^m |(D \tilde{g})(A_l^{\varepsilon})|^2+{1\over 2}\sum_{l=1}^m (D^2\tilde{g})(A_l^{\varepsilon}, A_l^{\varepsilon})
+(D\tilde{g})(A^{\varepsilon}_0)\\
&&{1\over 2}\sum_{l=1}^m |(D \tilde{g})(A_l^{N,\varepsilon})|^2+{1\over 2}\sum_{l=1}^m (D^2\tilde{g})(A_l^{N,\varepsilon}, A_l^{N,\varepsilon})
+(D\tilde{g})(A^{N,\varepsilon}_0)
\end{eqnarray*}
are bounded above with the upper bound  uniform in $\varepsilon\in (0, \varepsilon_0)$ and in $N>0$. From the calculations in  Lemma 6.1 and Theorem 5.1 of \cite{Li1} or see Theorem \ref{old-thm} before, for every $p>0$, 
\begin{equation}\label{g5}
\sup_{0<\varepsilon<\varepsilon_0,N>0}\sup_{x \in S}\left(   \E\sup_{0\leqslant t \leqslant T} 
|V_t^{N,\varepsilon}(x)|^p+|V_t^{\varepsilon}(x)|^p\right)\leqslant C< \infty 
\end{equation} 
%|\xi_t^{N,\varepsilon}(x)|^p++|\xi_t^{\varepsilon}(x)|^p
%is bounded by a constant which we denote by $L(p,S,T)$.  This and (\ref{g3}),
As we remark before, the same theory can apply to SDE (\ref{e1}) and (\ref{e2}) with 
strong solution, we obtain, 

\begin{equation}\label{g8}
\sup_{N>0}\sup_{x \in S}\left(   \E\sup_{0\leqslant t \leqslant T} 
|V_t^{N}(x)|^p+|V_t(x)|^p\right)\leqslant C< \infty 
\end{equation} 

As before, let $T_{N}(x)$, $T_{N}^{\varepsilon}(x)$ be the first exit times from the ball $B_N$  of $\xi_{\cdot}(x)$ and $\xi_{\cdot}^{\varepsilon}(x)$. For  $x \in S$, $0<\varepsilon<\varepsilon_0$ and $N$ large so that  $S\subseteq B_N$, we have,
\begin{equation}\label{g11}
\begin{split}
&\E\sup_{0\leqslant t \leqslant T}
|V_t^{N,\varepsilon}(x)-V_t^{\varepsilon}(x)|^p\\
&\leqslant
C\E\Big[\Big(\sup_{0\leqslant t \leqslant T}|V_t^{N,\varepsilon}(x)|^p
+\sup_{0\leqslant t \leqslant T}|V_t^{\varepsilon}(x)|^p\Big)I_{\{T>T_{N}^{\varepsilon}(x)\}}
\Big]\\
&\leqslant C\Big(\sqrt{\E\sup_{0\leqslant t \leqslant T}|V_t^{N,\varepsilon}(x)|^{2p}}
\sqrt{\mathbf{P}(T>T_{N}^{\varepsilon}(x))}+\sqrt{\E\sup_{0\leqslant t \leqslant T}|V_t^{\varepsilon}(x)|^{2p}}
\sqrt{\mathbf{P}(T>T_{N}^{\varepsilon}(x))}\Big)\\
&\leqslant C\big(\mathbf{P}(T>T_{N}^{\varepsilon}(x))\big)^{1/2}\\
&\leqslant  \frac{C\sup_{0<\varepsilon<\varepsilon_0,N>0}\sup_{x \in S} \sqrt{\E\sup_{0\leqslant s \leqslant T} 
|\xi_s^{N,\varepsilon}(x)|^{2p}}}{N^p}
\leqslant \frac{C}{N^p} 
\end{split}
\end{equation}
%The uniform exit time estimate (\ref{g12}) in the case of $\Psi$ is not of linear growth, or the corresponding estimates in the case
%of linear growth gives: 
%\begin{equation}\label{g12}
%\begin{split}
%\mathbf{P}(T>T_{N}^{\varepsilon}(x))\leqslant 
%\mathbf{P}(\sup_{0\leqslant t \leqslant T}|\xi_t^{\varepsilon}(x)|\geqslant N)
%\leqslant \frac{L(p,S,T)}{N^p}
%\end{split} 
%\end{equation}
%By (\ref{g11}) and (\ref{g12}), we get,
Here in the last step we use the estimation (\ref{l1}) by linear growth condition of $A_l$. So
by (\ref{g11}), we get, 
\begin{equation*}
\lim_{N\rightarrow \infty}\sup_{0<\varepsilon<\varepsilon_0}\sup_{x \in S}\E\sup_{0\leqslant t \leqslant T}
|V_t^{N,\varepsilon}(x)-V_t^{\varepsilon}(x)|^p=0. 
\end{equation*}
Analogously, using (\ref{g8}), we have  the results for the quantities without $\varepsilon$.

\end{proof}

 In the proof of  Theorem 3.1,  ellipticity condition is only needed for the estimate of the item 
$\mathbb{P}(\xi_s(x)\in U(R,\zeta))$ and Lemma \ref{lemma2} holds automatically if $A_l$ are $C^1$.
The corresponding theorem for non-elliptic systems are given below.
\begin{proposition}
\label{C1-case}
Suppose the coefficients $A_l$ of SDE (\ref{e1}) are $C^1$ and satisfies Assumption \ref{J1}. Then
$$\lim_{\varepsilon\rightarrow 0}\sup_{x \in S}\E
\sup_{0\leqslant s \leqslant T }|V_s^{\varepsilon}(x)-V_s(x)|^p=0$$
for each $T>0$, $p>0$ and bounded subset $S$ in $\R^n$. 
\end{proposition}

\section{Regularity of the solution flow}
\label{regularity}
Theorem 4.5.1 in \cite{KU} states that for SDE (\ref{e1}), with $A_l$  global Lipschitz continuous, 
there is a solution flow $\xi_t(x,\omega)$ such that for almost surely all  $\omega$ 
and every $t>0$, $\xi_t(\cdot,\omega) \in C^{0,\delta}(\R^n)$($0<\delta<1$).
See \cite{Li1,Fang,Zh1} for various generalisation. 
%are generalised to SDE with locally Lipschitz continuous coefficients or non-Lipschitz continuous coefficients.
%Here we  use the approximation theorem to prove the following results.
To our knowledge the following result on solution with Sobolev regularity is new.
\begin{thm}\label{th:regularity}
Assume Assumption \ref{N1}  and %Condition $G(\sigma, T_0)$ holds for $G$ 
%the functions $F_i$ defined in Remark \ref{remark-G}. 
condition (\ref{N17}) hold. There is a global solution flow $\xi_{t}(x,\omega)$ for SDE (\ref{e1}), \i.e. a version such that  for almost surely all  $\omega$,  
$\xi_{\cdot}(\cdot,\omega)$ is continuous in $[0,\infty)\times \R^n$. Furthermore  for each $p>0$ , there is a constant $T_5(K,\alpha, M,\theta,n,p, T_0,\tilde{T}, \sigma)$, such that 
$\xi_t(\cdot, \omega) \in W^{1,p}_{\loc}(\R^n)$ for each $0<t\leqslant T_5$.
\end{thm}

\begin{proof}

From the analysis in the proof of Theorem 4.1 in \cite{Li1} for SDE (\ref{e2}) with smooth coefficients, 
given a bounded set $S$ in $\R^n$, we have for each $x,y \in S$ and $T>0,p\geqslant 1$, 
\begin{equation*}
\E\sup_{0\leqslant s \leqslant T}|\xi_s^{\varepsilon}(x)-\xi_s^{\varepsilon}(y)|^p
\leqslant |x-y|^p\sup_{z \in S} \E\sup_{0\leqslant s \leqslant T}|V_s^{\varepsilon}(z)|^p
\end{equation*}
By Lemma \ref{N2} and Lemma \ref{Lem N4}, let $\varepsilon$ tend to 0, there exists a $\hat{T}>0$
which only depends on $K,\alpha, M,\theta,n,\tilde{T},T_0,p,\sigma$, such that,
\begin{equation*}
\E\sup_{0\leqslant s \leqslant \hat{T}}|\xi_s^{\varepsilon}(x)-\xi_s^{\varepsilon}(y)|^p
\leqslant C |x-y|^p 
\end{equation*}
and from that one note that
$$\E|\xi_t(x)-\xi_s(y)|^p
\leqslant C \big(|x-y|^p+|t-s|^{\frac{p}{2}}\big) \ \  \ 0\leqslant t,s, \leqslant \hat{T},\ x,y \in S$$
So in above estimate, we take $p>n$, then by the Kolmogorov's criterion, there
is a version of the solution flow $\xi_t(x,\omega)$ for SDE (\ref{e1}), such that
$\xi_{.}(\cdot,\omega)$ is continuous in $[0,\hat{T}]\times \R^n$. As for $t>\hat{T}$, note that
by the uniqueness of the strong solution of SDE \ref{e1} under Assumption \ref{N1}, it is
satisfied that 
\begin{equation}\label{l5}
\xi_{t}(x,\omega)=\xi_{t-\hat{T}}(\xi_{\hat{T}}(x,\omega),\theta_{\hat{T}}(\omega)) \ \ a.s. 
\end{equation}
where $\theta_{\hat{T}}(\omega)_t=\omega_{t+\hat{T}}-\omega_{\hat{T}}$ is the shift operator 
and hence  the solution flow $\xi_{\cdot}(\cdot,\omega)$ is continuous in $[0,2\hat{T}]\times \R^n$ and hence  on $[0,\infty)\times \R^n$ by repeating the procedure.

By Lemma \ref{N2}, Theorem \ref{Th2} and the diagonal principle there exist a constant $T_5>0$, a subsequence $\varepsilon_k$
with $\lim_{k\rightarrow \infty} \varepsilon_k=0$ and a set
$\tilde{\Lambda}_1$ with $\mathbb{P}(\tilde{\Lambda}_1)=0$, such that if $\omega \in \tilde{\Lambda}_1^c $, for all $N>0$,
\begin{equation}\label{e36}
\lim_{k\rightarrow \infty}\int_{|x|\leqslant N}\sup_{0\leqslant t \leqslant T_5}
|V_t^{\varepsilon_k}(x,\omega)-V_t(x,\omega)|^pdx=0
\end{equation}
  and 
\begin{equation}\label{e37}
\lim_{k\rightarrow \infty}\int_{|x|\leqslant N}\sup_{0\leqslant t \leqslant T_5}
|\xi_t^{\varepsilon_k}(x,\omega)-\xi_t(x,\omega)|^p dx=0.
\end{equation}
Here $V_t(x),\ t>0$ is the solution of SDE (\ref{e4}) 
we get in the Theorem \ref{Th2}. 
 
Let $\{e_r\}$ be an o.n.b. of $\R^n$.  For simplicity write $V_t^{k,r}(x)=\langle V_t^{\varepsilon_k}(x),e_r
\rangle_{\R^n}$ and
$\xi_t^{k}(x)=\xi_t^{\varepsilon_k}(x)$. For the SDE (\ref{e2}) whose coefficients are smooth 
and with bounded derivatives, 
there is a smooth solution flow $\xi_t^k(\cdot,\omega)$ and $\frac{\partial \xi_t^k(\cdot,\omega)}
{\partial r}=V_t^{k,r}(x,\omega)$. Therefore there  exists  a null set $\Lambda_k$,  such that 
if $\omega \not  \in \Lambda_k$, the following integration by parts formula 
holds for $0\leqslant t \leqslant T_5$ and any $\varphi \in C_0^{\infty}(\R^n)$, 
\begin{equation}\label{e38}
\int_{\R^n}{\partial \varphi \over \partial x_ r}(x)  \xi_t^k(x,\omega)dx
=-\int_{\R^n}\varphi(x)V_t^{k,r}(x,\omega)dx.
\end{equation}
Let $\tilde{\Lambda}:=(\bigcup_{k=1}^{\infty}\Lambda_k)\cup\tilde{\Lambda}_1$, a null set.
Taking the limit $k\to \infty$ in the above identity and using (\ref{e36}-\ref{e37}) to see 
when $\omega \not  \in \tilde{\Lambda}$ and $0\leqslant t \leqslant T_5$, 
\begin{equation}\label{e39}
\int_{\R^n}  {\partial \varphi \over \partial x_r}(x)\xi_t(x,\omega)dx
=-\int_{\R^n}\varphi(x)V_t^{r}(x,\omega)dx  
\end{equation}
which means that $\xi_t(\cdot,\omega)$  is weak differentiable in 
distribution sense for  almost surely all $\omega$ and 
$\frac{\partial \xi_t(\cdot,\omega)}
{\partial x_r}=V_t^{r}(\cdot,\omega)$. Next we prove 
$\xi_t(\cdot,\omega) \in W^{1,p}_{\loc}(\R^n) $ for each $p>0$. For $N>0$, 
\begin{equation*}
\E\int_{B_N}\sup_{0\leqslant t \leqslant T_5}|V_t^r(x,\omega)|^pdx 
=\int_{B_N}\E\sup_{0\leqslant t \leqslant T_5}|V_t^r(x,\omega)|^pdx
\leqslant C\lambda(B_N) 
\end{equation*}
Hence $\int_{B_N}\sup_{0\leqslant t \leqslant T}|V_t^r(x,\omega)|^pdx $ is finite almost surely
and so we can find a null set $\Gamma_1$, such that 
$\int_{B_N}\sup_{0\leqslant t \leqslant T}|V_t^r(x,\omega)|^pdx <\infty$ for every 
$N>0$ when $\omega \notin \Gamma_1$. As the same way, we can prove the similar property for $\xi_t(x,\omega)$.
%Let $\Gamma^{N,p}=\big\{\omega:\ \sup_{0\leqslant t\leqslant T}\int_{B_N}
%|V_t^r(x,\omega)|^pdx <+\infty\big\}$ and $\Gamma:=\cup_{p=1}^{\infty}
%\cup_{N=1}^{\infty}
%\Gamma^{N,p}$. Then 
%$$\mathbb{P}(\Gamma)\leqslant \sum_{p=1}^{\infty}\sum_{N=1}^{\infty}
%\mathbb{P}(\Gamma^{N,p})=
%=0. $$
%In particular, if $\omega \in \Gamma^c$, for 
%every $N>0, p>0$ and $t$, $\int_{B_N}|V_t^r(x,\omega)|^p dx$ is finite, which 
%means 
Hence $\xi_t(x,\omega), V_t^r(x,\omega) \in L^p_{\loc}(\R^n)$ for $\omega \not \in \Gamma\cup \tilde{\Lambda}$ where $\Gamma$
is a set with measure $0$, which means almost surely
$\xi_t(\cdot, \omega) \in W^{1,p}_{\loc}(\R^n)$ for each $0<t\leqslant T_5$. 
%Let $\Lambda:=\tilde{\Lambda}\cap \Gamma$, we have $\mathbb{P}(\Lambda)=0$ and (\ref{e39}) holds and $V_t^r(x,\omega) \in L^2_{\loc}(\R^n)$ for $\omega \in \Lambda^c$, which means we have finished the proof.  

\end{proof}

For SDE with locally Lipschitz continuous coefficients, we may get rid of the boundedness and the uniform ellipticity condition.

\begin{thm}\label{Th4}
If  Assumption \ref{J1} holds and the coefficients are elliptic,  there is a 
 solution flow $\xi_t(x,\omega)$ to SDE (\ref{e1}) with the property that 
$(t,x)\mapsto \xi_t(\cdot,\omega)$  is continuous  for almost surely all  $\omega$ and 
$\xi_t(\cdot, \omega) \in W^{1,p}_{\loc}(\R^n)$ for each $t>0, p>0$.
\end{thm}

\section{The differentiation formula }
\label{bel-formula}
The  uniqueness of a strong solution of SDE (\ref{e1}) leads to
 the Markov property of the solution $\xi_t(x)$, see  a proof in \cite{Oks} that can be easily followed under  Assumption \ref{N1}. For $f \in {\mathcal B}_b(\R^n)$ define $P_tf(x):=\E f(\xi_t(x))$ so $P_t$  is the associated Markov semigroup to (\ref{e1}). In this section a representation for $dP_t$
is given which leads to an integration by parts formula for the measure induced by the solution of the SDE (\ref{e1}). 
%In this section we will use the approximation Theorem \ref{th-derivative}  to prove the following Bismut-Elworthy-Li
%formula for SDE (\ref{e1}). 
Let $\xi_t(x)$  be the solution flow of SDE (\ref{e1}) and
$V_t(x) \in L(\R^n,\R^n)$ the solution of (\ref{e3}) constructed in Theorem \ref{Th2}. Let
 $V_t(x,v_0)=\langle V_t(x),v_0\rangle_{\R^n}$ for $v_0 \in \R^n$ and $Y:\mathbf{R}^n\rightarrow L(\mathbf{R}^n,\R^m)$  the right inverse of map $A:\R^n\rightarrow L(\R^m,\R^n)$, where 
\begin{equation}\label{J3}
A(x)(a):=\sum_{l=1}^ m a_lA_l(x)  \qquad \text{for } a=(a_1,a_2,...,a_m) \in \R^m.                                                   \end{equation}

\begin{thm}\label{P1}
Suppose the Assumption \ref{N1} and condition (\ref{N17}) hold, 
% Condition $G(\sigma,T_0)$ holds for the functions $F_i$, 
then there is a constant
$T_{6}(K,\alpha, M,\theta,n,T_0,\tilde{T},\sigma)$, such that

\begin{equation}\label{e23}
d P_t f(x)(v_0)
=\frac{1}{t}\E \big[f(\xi_t(x))\int_{0}^{t}
\langle Y(\xi_s(x))(V_s(x,v_0)), dW_s\rangle_{\R^m} \big], \ \ v_0\in \R^n
\end{equation} 
for any $f$ in $ {\mathcal B}_b(\R^n)$ and $0<t\leqslant T_{6}$. If moreover 
$f \in C_b^1(\R^n)$, then
 $d(P_tf)(x)(v_0)=\E df(V_t(x,v_0)), $ for all $\ v_0 \in \R^n$ and such $t$. 
%0\leqslant t \leqslant T_{6}$. {\bf ????check, not quite right}
\end{thm}

\begin{proof} 
Take $f\in C^1_b(\R^n)$. Since the coefficient of SDE (\ref{e2}) is smooth and with
bounded derivatives, by the classical 
formula in \cite{ELi1}, we have
\begin{equation}\label{l8}
d \E f(\xi_t^{\varepsilon}(x))(v_0))=\E df(V_t^{\varepsilon}(x,v_0)), \ v_0 \in \R^n 
\end{equation}
and
\begin{equation}\label{e21}
d \E f(\xi_t^{\varepsilon}(x))(v_0)
=\frac{1}{t}\E \big[f(\xi_t^{\varepsilon}(x))\int_{0}^{t}\langle
 Y^{\varepsilon}(\xi_s^{\varepsilon}(x))
(V_s^{\varepsilon}(x,v_0)), dW_s\rangle_{R^m} \big]
\end{equation}
where $Y^{\varepsilon}:\R^n\rightarrow L(\R^n,\R^m)$ is the right inverse of map $A^{\varepsilon}:\R^n\rightarrow L(\R^m,\R^n)$. 
%The L.H.S.  of (\ref{e21}) converges as $\varepsilon$ converges to zero since
%$d \E f (\xi_t^\varepsilon(x))(v_0)=\E df (T\xi_t^\varepsilon (v_0))\to \E df (T\xi_t(v_0))$.
Since $f\in C^1_b(\R^n)$, by Lemma \ref{N2} and Theorem \ref{Th2}, there is a constant
$T_6>0$, such that for any bounded set $S$ in $\R^n$, 
\begin{equation}\label{l6}
\lim_{\varepsilon\rightarrow 0}\sup_{x \in S}
\E\sup_{0\leqslant s \leqslant T_6}|f(\xi_s^{\varepsilon}(x))-f(\xi_s(x))|^4=0
\end{equation}
and
\begin{equation}\label{l7}
\lim_{\varepsilon\rightarrow 0}\sup_{x \in S}\E\sup_{0\leqslant s \leqslant T_6}
|V_s^{\varepsilon}(x,v_0)-V_s(x,v_0)|^4=0 
\end{equation}

So by (\ref{l8}) and (\ref{l6}-\ref{l7}) we obtain that 
$P_t f$ is differentiable if $f\in C^1_b(\R^n)$ and 
$d(P_tf)(x)(v_0)=\E df(V_t(x,v_0)), \ v_0 \in \R^n$
for each $0\leqslant t \leqslant T_6$.

For  the square matrix $(A^\varepsilon)^{\ast}A^\varepsilon$($\ast$ here means 
the transpose of the matrix), related to $A^\varepsilon$, with entries $(a_{ij}^\varepsilon)$, we have
 $$\left ( (A^\varepsilon)^{\ast}A^\varepsilon\right )^{-1}\to \left ( A^{\ast}A\right )^{-1}$$ as the inverse map is smooth in the Lie group of $n\times n$ nonsingular matrices. Consequently
$Y^\varepsilon=(A^\varepsilon)^{\ast}\left ( (A^\varepsilon)^{\ast} A^\varepsilon\right )^{-1}\to Y$.
In fact let $A_{il}^{\varepsilon}$, $a_{ij}^{\varepsilon}$ be defined as above, 
then for any $\eta=(\eta_1,\eta_2,...,\eta_n) \in \R^n$,
$Y^{\varepsilon}(x)(\eta)=
(\zeta_1^{\varepsilon}(x),\zeta_2^{\varepsilon}(x)...,\zeta_m^{\varepsilon}(x))$ 
where $\zeta_l^{\varepsilon}(x)=\sum_{i,j=1}^nA_{il}^{\varepsilon}(x)
b_{ij}^{\varepsilon}(x)\eta_j$. Here
$$(b_{ij}^{\varepsilon}(x))=(A^\epsilon(x))^{-1}.$$
Since by Assumption \ref{N1}, for $\varepsilon$ small, $a_{ij}^\varepsilon$ have the same uniform elliptic
constants and $A_l^{\varepsilon}$ are uniformly bounded with $\varepsilon$, we obtain that
$b_{ij}^{\varepsilon}$ converges to $b_{ij}$ uniformly in  $\R^n$,
and $b_{ij}^{\varepsilon}$ are uniformly bounded with $\varepsilon$. 
Also note that the bounded, uniformly continuity of $A_l$ and the uniformly positive lower bounded 
of the determination of matrix $A^{\ast}A$ (defined as (\ref{J3})) in $\R^n$ implies that
$b_{ij}$ is uniformly continuous in $\R^n$. So by Lemma \ref{N2} we can show that for any $T>0$ and bounded set $S$ 
in $\R^n$, 
\begin{equation*}%\label{e17}
\lim_{\varepsilon\rightarrow 0}\sup_{x \in S}\E\sup_{0\leqslant s \leqslant T}|b_{ij}^{\varepsilon}
(\xi_s^{\varepsilon}(x))-b_{ij}(\xi_s(x))|^4=0.
\end{equation*}
This together with the convergence (\ref{l7}) leads to
\begin{equation*}%\label{e19}
\lim_{\varepsilon\rightarrow 0}\sup_{x \in S}\int_0^{t}
\E|Y^{\varepsilon}(\xi_s^{\varepsilon}(x))(V_s^{\varepsilon}(x,v_0))
-Y(\xi_s(x))(V_s(x,v_0))|^2 ds=0
\end{equation*}
for all $t \leqslant T_6$ and  $S\subset \R^n$ bounded. 
Then by (\ref{l6}), we see that,
\begin{equation}\label{e22}
\begin{split}
&\lim_{\varepsilon\rightarrow 0}\sup_{x \in S}\Big|\E \big[f(\xi_t^{\varepsilon}(x))\int_{0}^{t}
\langle Y^{\varepsilon}(\xi_s^{\varepsilon}(x))(V_s^{\varepsilon}(x,v_0)), dW_s\rangle_{\R^m} \big]\\
&\qquad  \quad -\E \big[f(\xi_t(x))\int_{0}^{t}\langle Y(\xi_s(x))(V_s(x,v_0)), dW_s\rangle_{\R^m} \big]\Big|=0.
\end{split}
\end{equation}
which implies the differentiation formula (\ref{e21}) holds for each 
$f \in C_b^1(\R^n)$.
For $f\in {\mathcal B}_b(\R^n)$. Let $f_N$ be a sequence in  $ C_b^1(\R^n)$ 
with $$ \lim_{N\rightarrow \infty}\int_{S}|f_N(x)-f(x)|^4 dx=0 $$
for all bounded set $S$ in $\R^n$. 
By the heat kernel upper bound  $p_t(x,y)\leqslant c_1t^{-n/2}e^{|x-y|^2 \over c_2 t}$ from  Lemma \ref{Lem N2}, 
$$\lim_{N\rightarrow \infty}\sup_{x \in S}\E|f_N(\xi_t(x))-f(\xi_t(x))|^4=0$$ 
also holds for each $0<t\leqslant T_6$ and this completes the proof.

%Since each $f_N$ satisfies (\ref{e23}), (\ref{e24}) completes the proof.
\end{proof}

\begin{thm}\label{th-formula}
Suppose that Assumption \ref{J1} holds and 
\begin{equation}\label{e29}
\sum_{i,j=1}^n a_{ij}(x)\xi_i \xi_j\geqslant  {1\over\ \e^{\sigma \psi_1(|x|)}} |\xi|^2 .
\ \ \ \forall x \in \R^n,\ \xi=(\xi_1,...,\xi_n)\in \R^n.
\end{equation}
Here $\sigma$ is some constant and $\psi_1:\R^+ \rightarrow \R^+$ is the function
in Assumption \ref{J1}, then for  any $f\in {\mathfrak B}_b(\R^n)$ and $t>0$, 
 \begin{equation}\label{BEL-formula}
d P_t f(x)(v_0)
=\frac{1}{t}\E \big[f(\xi_t(x))\int_{0}^{t}
\langle Y(\xi_s(x))(V_s(x,v_0)), dW_s\rangle_{\R^m} \big], \ \ v_0\in \R^n
\end{equation} 
\end{thm}

\begin{proof}
In the proof the constants $C$ amy change in different lines and do not
depend on $N$. As the same approximation methods in the above Theorem 
and by Theorem \ref{Th3}, we can prove that if the coefficients
of SDE (\ref{e1}) are bounded, uniform elliptic and Lipschitz continuous, then the diffrentiation formula 
(\ref{BEL-formula}) holds for any  $f\in {\mathfrak B}_b(\R^n)$ and $t>0$. 
Let $f^N, A_l^N$ be the cut-off functions defined by (\ref{h1}) and that $\xi_t^N(x)$, $V_t^{N}(x)$ the solution of 
corresponding SDE.  So by the analysis above, $P_t^N f(x)=\E(f(\xi_t^N(x)))$ is differentiable with $x$, and 
for any  $f\in {\mathfrak B}_b(\R^n)$ and $t>0$, 
\begin{equation}\label{e30}
d \E f(\xi_t^{N}(x))(v_0)
=\frac{1}{t}\E \big[f(\xi_t^{N}(x))\int_{0}^{t}\langle
 Y^{N}(\xi_s^{N}(x))
(V_s^{N}(x,v_0)), dW_s\rangle_{R^m} \big] 
\end{equation}
where $Y^N$ is a right inverse of $A^N$. 

By the elliptic condition  (\ref{e29}) and the 
expression of $Y$ we use in the proof of Theorem \ref{P1}, there are constants $C>0$, $k \in \mathbb{N}^+$, such that
$|Y(x)|\leqslant C\e^{\sigma \psi_1(|x|)}\big(\psi_2(|x|)\big)^k$, where $\psi_i, i=1,2$ are the functions
as that in Assumption \ref{J1}, and the same estimate also holds for $Y^N$. 
Let $T_N(x)$ be the first exit time of $\xi_t(x)$ from the ball $B_N$, then  
for each $T>0$, $p>0$,  
\begin{equation*}
\begin{split}
&\sup_{x \in S}\E\sup_{0\leqslant s \leqslant T}
|Y^N(\xi_s^{N}(x))-Y(\xi_s(x))|^p\\
&\leqslant C \sup_{x \in S}\E \Big[ \sup_{0\leqslant s \leqslant T}\big(\exp\{( \sigma p \psi_1( |\xi_s^{N}(x)|)\})
\big)
\big(\psi_2( |\xi^N_s(x)|)\big)^{kp}I_{T>T_N(x)}\Big]\\
&\leqslant C \sup_{x \in S} \E\Big[\sup_{0\leqslant s \leqslant T}
\Big(\exp\{( 2\sigma p \psi_1( |\xi_s^{N}(x)|)\}+\big(\psi_2(|\xi^N_s(x)|)\big)^{2kp}
\Big)I_{T>T_N(x)}\Big]
\end{split} 
\end{equation*}
By the analysis in Section 3 (see also Theorem 5.1 in \cite{Li1}), if Assumption
\ref{J1} holds, 
we have, 
\begin{equation}\label{e31}
\sup_N\sup_{x \in S} \E\Big[\sup_{0\leqslant s \leqslant T}
\Big(\exp\{( 2\sigma p \psi_1( |\xi_s^{N}(x)|)\}+\big(\psi_2(|\xi^N_s(x)|)\big)^{2kp}
\Big)^2\Big]<\infty 
\end{equation}
and
\begin{equation}\label{e34}
 \sup_{x \in S} \p(T>T_N(x))\leqslant 
\frac{\sup_{x \in S} \E\sup_{0\leqslant s \leqslant T}|\xi_s^N(x)|^2}{N^2}
\leqslant \frac{C}{N^2} 
\end{equation}
By (\ref{e31}) and (\ref{e34}), it follows that
%%By (\ref{e31}) and (\ref{e34})
\begin{equation}\label{e35}
\lim_{N\rightarrow \infty}
\sup_{x \in S}\E\sup_{0\leqslant s \leqslant T}
|Y^N(\xi_s^{N}(x))-Y(\xi_s(x))|^p=0 
\end{equation}

%This quantity converges to $0$ as $N\to \infty$  since,  Theorem 5.1 \cite{Li1}, 
%\begin{eqnarray*}\label{e31}
% \sup_N\sup_{x \in S}\E\sup_{0\leqslant s \leqslant T} \left(   g(|\xi_s^\varepsilon)+g(|\xi_s|)|V_s^{N}(x)|^p+|V_s(x)|^p \right)
%<\infty,&&\\
%\sup_{x \in S} \mathbf{P} (T>T_N(x))\leqslant \frac{L(p,S,T)}{N^p}. \hskip 155pt &&
%\label{e34}
%\end{eqnarray*}
%  \exp({c\psi_2(|\xi_s^{N}(x)|}+ \exp({c\psi_2(|\xi_s(x)|))}+
Also note that from the analysis of Section 3, if Assumption
\ref{J1} holds, then, 
\begin{equation} \label{e32}
\lim_{N\rightarrow \infty}\sup_{x \in S}\E\sup_{0\leqslant s \leqslant T}\left( |\xi_s^{N}(x)-\xi_s(x)|^p+|V_s^{N}(x)-V_s(x)|^p\right)=0
\end{equation}

By (\ref{e35}) and (\ref{e32}), (\ref{BEL-formula}) holds for  each $f\in C_b^1(\R^n)$.
For $f\in {\mathfrak B}_b(\R^n)$, take an approximating sequence 
$f_{\varepsilon} \in C_b^1(\R^n)$, such that for any bounded set $S$ in $\R^n$, 
$$\lim_{\varepsilon\rightarrow 0}\int_S |f_{\varepsilon}(x)-f(x)|^pdx=0$$
and $||f_{\varepsilon}||_{L^{\infty}}\leqslant ||f||_{L^{\infty}}$. 
Note that $\xi_t^N(x)$ is the solution of a SDE with unformly elliptic, global Lipschitz continuous and bounded 
coefficients, so by the Markov kernel estimate, we have for each fixed $N$ and $t>0$, 
\begin{equation}\label{h2}
\lim_{\varepsilon\rightarrow 0}\sup_{x \in S}\E|f_{\varepsilon}
(\xi_t^N(x))-f(\xi_t^N(x))|^p=0
\end{equation}
And then
\begin{equation*}
\begin{split}
&\sup_{x \in S}\E|f_{\varepsilon}
(\xi_t(x))-f(\xi_t(x))|^p 
\leqslant C\big[\sup_{x \in S}\E|f_{\varepsilon}(\xi_t^N(x))-f_{\varepsilon}(\xi_t(x))|^p\\
&+\sup_{x \in S}\E|f(\xi_t^N(x))-f(\xi_t(x))|^p+\sup_{x \in S}\E|f_{\varepsilon}
(\xi_t^N(x))-f(\xi_t^N(x))|^p\big]\\
&\leqslant C\big[||f||_{L^{\infty}}\sup_{x \in S}\mathbf{P}(T>T_N(x))+
\sup_{x \in S}\E|f_{\varepsilon}(\xi_t^N(x))-f(\xi_t^N(x))|^p\big]
\end{split}
\end{equation*}
By (\ref{e34}) and (\ref{h2}), in above inequality first let $\varepsilon\rightarrow 0$, then
$N\rightarrow \infty$, we obtain,
\begin{equation}\label{h4} 
\lim_{\varepsilon\rightarrow 0}
\sup_{x \in S}\E|f_{\varepsilon}
(\xi_t(x))-f(\xi_t(x))|^p=0
\end{equation}
The proof is complete.     
\end{proof}

 \subsection{%Malliavin derivative 
Integration by parts formula}
 Let $H=L_0^{2,1}(\R^m)$ be the space of real valued function from $[0,T]$ to $\R^m$, starting from $0$ and
  with finite energy, which is also equipped with the usual Hilbert space structure.    Let $(\Omega, \F, P)$ be the standard
  Wiener space and $d$ the unbounded closed linear operator $ L^p(\Omega, \R) \to L^p(\Omega, L(H,\R)), \ p>1$ which agrees
  with the standard differentiation on $BC^1$ functions (see \cite{ELi3} and reference in that).  Let $\D^{1,p}$ be the domain of $d$, which is the closure of smooth cylindrical functions under the graph norm, and by tradiiton we denote the extension
by $T$.  
%Furthermore $T\xi_\cdot (h)=T\xi_t \int_0^t T\xi_s^{-1} (X(\xi_s)(\dot h_s))ds.$   If the coefficients are furthermore smooth then the solution belongs to $D^\infty$. 

Let $C_{x}([0,T];\R^n):=\big\{\gamma:\ \gamma \in C([0,T],\R^n), \gamma(0)=x \big\}$ and
\begin{equation*}%\label{l10}
\I:\Omega \rightarrow  C_{x}(\R^n)   \ \ \I(\omega)_t:=\xi_t(x,\omega) 
\end{equation*}
be the It\^o map, where $\xi_t(x,\omega)$ is the solution of SDE (\ref{e1}). 
It is standard result that  $\I_t:\Omega \rightarrow \R^n$  belong to the space of $D^{1,p}$ for all $p>1$ and $t$ if
the coefficients of SDE (\ref{e1}) are Lipschitz continuous (see \cite{Nua}). Furthermore, if the coefficients are 
smooth and with bounded derivatives, then by the results of Bismut, 
\begin{equation*}
T_{\omega}\I: H\rightarrow T_{\xi_{.}(x,\omega)}C_x([0,T];\R^n)  
\end{equation*}
the $H$ derivative for the It\^o map
$\I$ in the sense of Malliavin calculus exists in $L^p$ for each $p>1$, and for 
$v_t(\omega):=T_{\omega}\I_t(h)$,
\begin{equation*}
v_t(\omega)=V_t(x) \int_0^t V_s^{-1}(x) (A(\xi_s(x))(\dot h_s))ds. \ \ 0\leqslant t \leqslant T,\ \ h \in H
\end{equation*}
where $V_t(x) \in \R^{n\times n}$ is the derivative process satisfying (\ref{e4}), $V_t^{-1}$ is its inverse
and $A$ is defined as (\ref{J3}).
And by \cite{Nua}, $v_t(\omega):=T_{\omega}\I_t(h)$ also satisfies the following SDE 
\begin{equation}\label{l11}
v_t=0+ \sum_{l=1}^m \Big(\int_0^t DA_l(\xi_s(x))v_s dW_s^l+ \int_0^t A_l(\xi_s(x))\dot h_s ds\Big)
+\int_0^t DA_0(\xi_s(x))v_s ds
\end{equation}

   Define  $V^h(\xi_{\cdot})_t:=\langle V_t(x), h_t\rangle_{\R^n}$ and
\begin{equation}
\label{divergence}
\delta V^h_t(\xi_{\cdot}):=\int_0^t \langle Y(\xi_s(x))V^{\dot{h}}(\xi_{\cdot}(x))_s, dW_s\rangle _{\R^m} 
\end{equation}
where $\dot{h}$ means the derivative of $h$ with time and $Y$ is the right inverse
of $A$ defined in (\ref{J3}). By the approximation theorem we derive the following result.
 
\begin{thm}\label{ibp}
Suppose the Assumption \ref{N1} holds and there exist $\sigma >0$ and $T_0>0$, the 
condition (\ref{N17}) is satisfied, then there is a
constant $T_{8}>0$, such that for any $0\leqslant T\leqslant T_{8}$, we consider 
the path space $C_x([0,T]; \R^n)$, given a 
$h:[0,T]\times\Omega\rightarrow \R^n$, an adapted stochastic process with $h(\omega)\in L_0^{2,1}([0,T_8]; \R^m)$ a.s. 
and $\E(\int_0^{T_8}
|\dot{h}_s|^2ds)^{\frac{1+\beta}{2}}<\infty$ for some $\beta>0$, 
$$\E dF(V^h(\xi_{.}))=\E F(\xi_{.}(x))\delta V_{T}^h(\xi_{.}) $$
where $F$ is the $BC^1$ function on path space
$C_x([0,T]; \R^n)$.
\end{thm} 

\begin{proof} When the coefficients of
SDE (\ref{e1}) are smooth, it was shown in \cite{ELi2} that the differentiation formula 
for $P_tf$  leads to an integration by parts formula. The theorem there was given for 
compact manifolds. However this results and its proof remain valid for non-comapct 
manifold  if the differentiation formula for $P_t f$ holds as the proof only involves the 
Markov  property. Since the  coefficients of the approximate SDE (\ref{e2}) are 
smooth, uniformly elliptic and with bounded derivatives,
$$\E dF(V^{h,\varepsilon}(\xi_{.}^{\varepsilon}(x)))=\E F(\xi_{.}^{\varepsilon}(x))\delta 
V_{T}^{h,\varepsilon}(\xi_{.}^{\varepsilon}), $$
for any $T>0$ where $V^{h,\varepsilon}(\xi_{.}^{\varepsilon}(x))=\langle V_t^{\varepsilon}(x),h_t\rangle_{\R^n}$ and 
$$\delta 
V_T^{h,\varepsilon}(\xi_{\cdot}^{\varepsilon})=
\int_0^T\langle Y^{\varepsilon}(\xi_s^{\varepsilon}(x))
V^{\dot{h},\varepsilon}(\xi_{\cdot}^{\varepsilon}(x))_s, dW_s\rangle_{\R^m}.$$
Now by Theorem \ref{Th2} and the analysis before for the convergence $Y^{\varepsilon}$, the proof in completed.
\end{proof}

\begin{remark}
As the same cut-off methods in Section \ref{lip}. suppose Assumption \ref{J1} holds, we can also prove the same results
in Section 4 for any time interval $[0,T], T>0$ 
\end{remark}

\section{Appendix: The Geometry of Regularization}
Let ${\mathcal L}$ be a smooth elliptic second order differential operator without zero order term with $a=(a_{ij})$ the matrix representation of its second order part. The non-singular symmetric  matrix $a$ has a square root which can be chosen to be locally Lipschitz continuous,  see Stroock-Varadhan \cite{Stroock-Varadhan} and the book of  Ikead-watanabe\cite{IW}. Hence ${\mathcal L}$ has a H\"ormander form representation ${\mathcal L}={1\over 2}L_{A_l}L_{A_l}+L_{Z}$ where $A_l, Z$ are vector fields. This representation is far from being unique. Each representation produces a stochastic flow and corresponding geometry. We investigate the geometry and the properties of the stochastic flows for the decomposition involving 
non-global Lipschitz continuous vector fields.

Give $M=\R^n$  and the Riemannian metric $(a_{ij})^{-1}$ induced by the family of vector fields $(A_1, \dots, A_m)$. We consider $\R^n$  as a trivial manifold with a non-trivial Riemannian structure.  Uniform ellipticity condition and boundedness of the diffusion coefficients implies that the induced Riemannian metric is quasi-isometric to the Euclidean metric. The ellipticity condition (\ref{e29}) and 
Assumption \ref{J1} would mean the new Riemannian metric is `weakly' quasi isometric with the Euclidean metric.
For simplicity, from now on in this seciotn, we assume  the coefficients of SDE (\ref{e1}), 
$A_l\in C_b^1(\R^n), 1\leqslant l \leqslant m$, $DA_l, 1\leqslant l \leqslant m$ and
$A_0$ is bounded and 
(global) Lipschitz continuous in $\R^n$. We can also obtain the results under more general condition by the cut-off
methods used in Section \ref{lip}.

For each $e\in \R^m$, define $A(x)(e)=\sum A_l(x)\langle e, e_i\rangle e_i$ where $\{e_i\}$ is an o.n.b. of $\R^m$. In the case when $A_l$ are smooth  and elliptic, this induces a smooth Riemannian metric on $\R^n$ as well as an affine connection which is adapted to the metric such that  $(\breve\nabla_\cdot X)(e)(x)=0$ for all $e\in [\ker X(x)]^\perp$.  See the analysis in  Elworthy-LeJan-Li \cite{ELL}. This leads to a smooth decomposition of $(\tilde \paral_t(x_t))^{-1} W_t$, where $\tilde\paral_t$ is the stochastic parallel translation along the paths of $\xi_{.}$ defined using $\breve\nabla$, into the sum of two independent  Brownian motions in $\R^m$, one of which is intrinsic to $\xi_{\cdot}$.

   In the non-smooth case we discuss a smooth approximation which preserves much of the properties of the connection, which leads to a non-smooth Riemannian geometry. We use
the approximation argument to prove a intrinsic integration by parts formula. Stronger regularity 
on $A_l$, than in sectios \ref{section2} and \ref{lip},  are required.

For $(A_1^\varepsilon, \dots, A_m^\varepsilon)$,  the smooth elliptic approximations of $(A_1, \dots, A_m)$
there are the affine connection  $\breve \nabla^\varepsilon $ and its adjoint connection $\hat \nabla^\varepsilon$
 given by
$$\breve \nabla^\varepsilon_vU =A^\varepsilon(x) D(Y^\varepsilon(x)U)(v), \qquad v\in T_xM, U\in \Gamma TM.$$
and $$\hat \nabla^\varepsilon_U V=\breve \nabla^\varepsilon_V U +[U,V], \ U,V\in \Gamma TM.$$
In components this reads
$$(\breve \nabla^\varepsilon_vU)_k(x_0)=(DU_k)_{x_0}(v)+\sum_{j=1}^n \langle A^\varepsilon(x_0) D(Y^\varepsilon(x_0)(v,e_j), e_k\rangle U_je_k$$ 
where $(e_j)$ is the standard basis of $\R^n$ and $U=(U_1,\dots, U_n)$ and 
$DY^{\varepsilon}: \R^n\rightarrow L(\R^n\times\R^n;\R^m)$, see \cite{ELL} and we follow the tradition there and call it the LW connection.  The last term  in the equation can be written as  $\Gamma_{ij}^{\varepsilon,k} v_iU_j e_k$ where $\{\Gamma_{ij}^{\varepsilon,k} , 1\leqslant i,j, k \leqslant n\}$ is a family of real valued smooth functions.
In particular this is the unique connection such that $(\breve\nabla^\varepsilon_v A)_{x_0}=0$ for all $v\in [ \ker A^\varepsilon(x_0)]^\perp$ and $x_0\in \R^n$.

Given a vector field along a continuous curve there is the stochastic covariant differentiation 
with a fixed connection defined for almost surely all paths, given by
$ \frac{\hat D^\varepsilon}{dt} V_t=\hat \paral_t^ \varepsilon{d\over dt}(\hat\paral_t^\varepsilon)^{-1}V_t$ where $\hat\paral_t^\varepsilon$ is the stochastic parallel translation using the connection $\hat \nabla$. 

 \begin{proposition}
 \label{proposition4.1}
Assume the SDE (\ref{e1}) is uniformly elliptic and $A_l\in C_b^1(\R^n)$ for  $l=1,\dots, m$. 
Suppose that $DA_l$, $l=1,\dots, m$ and $A_0$ are bounded and  (global) Lipschitz continuous in $\R^n$. Let
$\breve\paral_s^\varepsilon: \R^n\to \R^n$ and $\hat \paral_s^\varepsilon:\R^n\to \R^n$ be the stochastic parallel translations
with the connection  $\breve \nabla^\varepsilon $ and $\hat \nabla^\varepsilon$  
respectively. Then $\breve\paral_s^\varepsilon: \R^n\to \R^n$ and $\hat \paral_s^\varepsilon:\R^n\to \R^n$ converge
in $L^p$ for any $p\geqslant 1$ and the martingale part of anti-stochastic development map also converges in  $L^p$ to a  Brownian motion $\breve B_t$. Furthermore
% Furthermore if  $(a_{ij})$ is uniformlly elliptic and globally Lipschitz continuous then the horizontal lift map converges.follows 
the filtration of $\{\breve B_s: 0\leqslant s \leqslant t\}$ is the same  as that of $\{\xi_s, 0\leqslant s \leqslant t\}$.
\end{proposition}
\begin{proof}
For any $v_0 \in \R^n$, let $\breve v^{\varepsilon}_t:=\breve\paral_s^\varepsilon(v_0)$ and 
$\hat v^{\varepsilon}_t:=\hat \paral_t^\varepsilon(v_0)$. Note that  the $k$-th component of such process satisfy 
$d\breve v_t^{\varepsilon,k}=-\Gamma_{i,j}^{\varepsilon,k}(\xi_t^{\varepsilon}) \breve v_t^{\varepsilon,j}
\circ d\xi_t^{\varepsilon,i}$ and $d\hat v_t^{\varepsilon,k}=-\Gamma_{j,i}^{\varepsilon,k}(\xi_t^{\varepsilon})  
\hat v_t^j\circ d\xi_t^{\varepsilon,i}$ 
respectively. For simplicity, we only prove the convergence of $\breve v_t$, and the same results can be proved 
for $\hat v_t$ as the same way. In fact, we have,  
\begin{equation}\label{A1}
d\breve v_t^\varepsilon=\sum_{l=1}^mG_l^\varepsilon(\xi_t^\varepsilon)(\breve v_t^\varepsilon)dW_t^l+G_0^\varepsilon(\xi_t^\varepsilon)(\breve v_t^\varepsilon)dt 
\end{equation}
where each $G_l^\varepsilon(x)$ for $l=1,\dots, m$, is a $m\times n$ matrix with the $(j,k)$ entry given by $\sum_{i=1}^nA_{il}^\varepsilon(x)\Gamma_{ij}^{\varepsilon,k}(x)$ and the drift term
$G_0$ is the sum of some items only involving $A_l$, $\Gamma_{ij}^{\varepsilon,k}(x)$ and their first
derivatives. Note that the Christoffel symbols are determined by $A^\varepsilon(DY^\varepsilon)$ (see the analysis in \cite{ELL}) and
$$DY^\varepsilon(v)
% =D[(A^\varepsilon)^T(A^\varepsilon(A^\varepsilon)^T)^{-1}](v)
=D(A^\varepsilon)^{\ast}(v)
+(A^\varepsilon)^{\ast} D((A^\varepsilon)^{*}A^\varepsilon)^{-1}(v).$$

From the assumption of the proposition we see that $G_l^{\varepsilon}$, $l=1,\dots, m$, 
are bounded in $\R^n$,  uniformly in $\epsilon$ for $\varepsilon$ sufficiently small. 
%Let $v_t^\varepsilon=\breve\paral_t^\varepsilon(v_0)$ then 
%$$dv_t^\varepsilon=\sum_{\ell=1}^mG_\ell^\varepsilon(x_t^\varepsilon)(v_t^\varepsilon)dB_t^\ell+G_0^\varepsilon(x_t^\varepsilon)(v_t^\varepsilon)dt$$
Let $\breve v_t$ be the solution to the corresponding SDE (\ref{A1}) without the $\varepsilon$ term and some
items related to the second order derivatives of $A_l$ are bounded modifications of the almost sure derivative 
of $D^2 A_l$, so the first derivatives of $\Gamma_{ij}^k$ makes sense. 
Then for the linear SDE (\ref{A1}), by a proof analogous to that of Lemma \ref{lemma3}, we have for each $t>0, p\geqslant 1$, 
\begin{equation*}
\lim_{\varepsilon \rightarrow 0}\E\sup_{0\leqslant s \leqslant t}|\breve v_s^{\varepsilon}-\breve v_s|^p=0 
\end{equation*}

%and
%\begin{eqnarray*}
%|v_t^p|&=&|v_0|^p+p\sum_\ell \int_0^t |v_s|^{p-2}\langle  DA_\ell (v_s), v_s\rangle dB_s^\ell+ p\int_0^t |v_s|^{p-2}  \langle DA_0(v_s), v_s\rangle  ds\\
%&&+ p\int_0^t   |v_s|^{p-2}\left( \sum_\ell |DA_\ell(v_s)|^2 +{(p - 2)\over 2}\sum_{\ell=1}^m \langle DA_\ell( v_s), v_s\rangle ^2 \right)ds.
%\end{eqnarray*}
%Together with the linear growth condition on $DA_\ell$ we see that $v_t^\varepsilon$ exists for all time and
%$$\E \sup_{s\le t} |v_s(x)|^p)\le |v_0|\E \sup_{s\le t} e^{\int_0^t C(1+|x_s|) ds}<\infty.$$
%The same  holds for the quantities with $\varepsilon$ hold.

%For the convergence we first assume that $DA_\ell$ are bounded and follow the proof of Theorem \ref{th-derivative}
%for the following estimate. Let $\alpha^\varepsilon(t,x)=C\sum_{\ell=1}^m  \int_0^t \E |(G_l^\varepsilon-G_\ell)(v_s(x))|^p\; ds$.
%$$\E \sup_{s\le t} |v_s^\varepsilon-v_s|^p
%\le \alpha^\varepsilon(t,x) +Ce^{Ct}  \int_0^t  \alpha^\varepsilon(s,x)ds.$$
%That the limit converges to zero follows from lemma \ref{lemma3}. We next follow the proof of Proposition \ref{J2} to apply a cut
%off procedure and conclude the limit from the moment estimates we obtained earlier.  The same argument works for $\hat \paral$ with $\Gamma_{ij}^k$ replaced by $\Gamma_{ji}^k$ in the proof.

%Let $\hat h^\varepsilon$ be the horizontal lift map determined by $\hat\nabla^\varepsilon$ then $\hat h_u(A(x)(e))={d\over dt}|_{t=0} (T\phi_t^e\circ u)$
%for any frame $u$ and $e\in \R^m$. Here $\phi_t^e$ satisfies ${d\over dt} \phi_t^e=A(\phi_t^e)(e)$. The converges of $\hat h^\varepsilon$ is immediate.
 
Let $\breve B_s^\varepsilon$ be the martingale part
of the stochastic anti-development map $\int_0^t  (\breve \paral_s^\varepsilon)^{-1}\circ d\xi_s^\varepsilon$.  
%$x_t^\varepsilon=\xi_t^\varepsilon(0)$. Each vector field  $A^\varepsilon_l$ has norm one under its intrinsic Riemannian metric and hence its Euclidean norm 
%also satisfies $|A^\varepsilon(x_s^\varepsilon)| $ is also bounded. 
Note that the stochastic parallel translation  $\breve \paral_s^\varepsilon$ is an isometry hence 
by the convergence results for $\breve \paral_s^\varepsilon$, it is straight forward  to 
show that $\breve B_s^\varepsilon$ converges in $L^p$ as $\varepsilon$ tends to $0$.
Since for each $\varepsilon$, $\breve B_s^\varepsilon$ is a Brownnian motion (see \cite{ELL}), 
so the limit process $\breve B_s$ is also a Brownnian motion. Moreover if $\breve \paral_s^{-1}$ is the inverse  of $\breve \paral_s$, the limit process of $\breve \paral ^{\varepsilon}_s$, $\breve B_s$ is the 
martingale part of $\int_0 ^t \paral^{-1}_s d \xi_s$. 
The Brownian motion $\breve B_\cdot$  is clearly adapted to the filtration of $\xi_\cdot$.  For the opposite inclusion of filtrations, let 
%$z_t^\varepsilon= (\breve \paral_s^\varepsilon)^{-1} \circ dx_t^\varepsilon$ and 
$h^\varepsilon_u( A_\ell^\varepsilon)$ be the horizontal lift  of $A_l^\varepsilon$ at frame $u$ and respect to $\breve \nabla^\varepsilon$ in the orthonormal frame bundle. Then the horizontal lift of the path $\xi_t^\varepsilon$ 
starting from the initial frame $u_0$ satisfies:
$$d \tilde \xi_t^\varepsilon = \sum_{l=1}^m h_{\tilde \xi_t^\varepsilon}  A_l^\varepsilon 
(\xi_t^\varepsilon)\circ d\breve B^{\varepsilon,l}_t +h_{\tilde \xi_t^\varepsilon} A_0^\varepsilon(\xi_t^\varepsilon) dt.$$
%Let $u$ be an initial frame at $x_0$ then $\breve{ \parals_t}^\varepsilon= \tilde x_t^\varepsilon  \circ u_0^{-1}$.
%$$f(z_t^\varepsilon)=f(z_0)+\int_0^t df_{z_s^\varepsilon} \left(  (\breve \paral_s^\varepsilon)^{-1} A_\ell(x_s^\varepsilon) dB_s^\ell
%+(\breve \paral_s^\varepsilon)^{-1} A_0(x_s^\varepsilon) ds\right)  $$
%and $z_t^\varepsilon=\breve B_t^\varepsilon+\int_0^t (\breve \paral_s^\varepsilon)^{-1} A_0^\varepsilon(x_s^\varepsilon) ds$, 
%$$\circ dx_t^\varepsilon = (\breve \paral_s^\varepsilon) \circ d\breve B_t^\varepsilon+ A_0^\varepsilon(x_t^\varepsilon) dt $$
%$$d  \tilde x_t^\varepsilon =h_{\tilde x_t}^\varepsilon \left( \tilde x_t^\varepsilon u_0^{-1} \circ  d\breve B_t+ A_0^\varepsilon(\pi(\tilde x_t^\varepsilon)) \; dt\right).$$
Note that the horizontal lift $h^{\varepsilon}$ only depends on the Christoffel symbols
$\Gamma_{ij}^{\varepsilon,k}(x)$ and $A_l^{\varepsilon}$. So take $\varepsilon \to 0$ to see the above equation 
also holds without parameter $\varepsilon$ (The second order derivatives of $A_l$ are viewed
as bounded version of the weak derivatives). And it implies that  $\tilde \xi_{.}$, therefore $\xi_{.}$ is adapted to the filtration of $\sigma\{\breve B_s: \}$ from the stochastic differential equation which defines it.
\end{proof}

The stochastic processes $V^h$ defined in Section \ref{bel-formula}  are not intrinsic objects on the path space. And use
the conclusion of Proposition \ref{proposition4.1}, we show  that 
$\E \{ T\I_t(h)| \sigma(\xi_s(x), 0\leqslant s \leqslant T)\}$ is  an intrinsic object  where $T\I(h)$ is the 
Malliavin derivative of the It\^o map $\I$.%, c.f. Proposition \ref{remark3}.
%Let $d\tilde B_t= \breve \paral_s d \breve B_t$.
\begin{thm}
Suppose the same assumption in Proposition \ref{proposition4.1} holds.
Let  $\hat W_t:\R^n \rightarrow \R^n$ be the damped parallel translation which satisfies the following equation
\begin{equation}\label{l13}
\frac{\hat D}{dt}[\hat W_t(v_0)] =-{1\over 2} (\breve \Ric)^{\#}(W_t(v_0))dt+
\breve \nabla_{W_t(v_0)}^\varepsilon A_0dt \ \ W_0(v_0)=v_0
\end{equation}
where $\frac{\hat D}{dt}=\hat \paral_t\frac{d}{dt}\big((\hat \paral_t)^{-1}\big)$ and $\Ric: \R^n\times \R^n\to \R$ is the Ricci tensor with the connection
$\breve \nabla$ and $\Ric^{\#}$ is the corresponding linear map on defined by $\breve \paral$ $\R^n$
(The $\breve \paral$, $\hat \paral_t$ are the limit process we get in Proposition \ref{proposition4.1}).
Then we have,
\begin{equation}\label{l12}
\E \{ T\I_t(h)| \sigma(\xi_s, 0\leqslant s \leqslant T)\}(\xi_{.})=\hat W_t \int_0^t (\hat W_s)^{-1}
A(\xi_s)\dot{h}_s ds
\end{equation}
\end{thm}
\begin{proof}
From equation (\ref{l11}), by the boundedness condition we have, it can be proved as before as that
$T\I^{\varepsilon}_t(h)$ is a Cauchy sequence in $L^p(\p)$ for any $p>0$, where $\I^{\varepsilon}$ is the Ito 
map defined by SDE (\ref{e2}). By the closibility of Malliavin derivative, we derive that $\I$ is differentiable in 
the Malliavin calculus.
By Theorem 3.3.7 in \cite{ELL}, for SDE (\ref{e2}) with smooth coefficients,
$\E\{ T\I^\varepsilon_t(h)| \sigma\{\xi^{\varepsilon}_s: 0\leqslant s\leqslant T\}\}$ satisfies
the equation (\ref{l12}) where process $\hat W_t^{\varepsilon}$ is defined similarly by 
equation (\ref{l13}). Note that the parallel translation in the equation (\ref{l13}) is defined by 
the adjoint connection, which is in general not adapted with some metric, so $\hat \paral$ is not a 
isometry in general. But we can tranform $\frac{\hat D}{dt}$ in the equation to 
$\frac{\breve D}{dt}$ defined by original LW connection, and the pointwise ODE (\ref{l13}) will become a
linear SDE with some torsion terms. Also note that we have the following formula for the curvature 
tensor (see \cite{ELL}), 
\begin{equation} \label{l14}
R^\varepsilon(u,v)(w)=\sum_i \breve \nabla_u A_i^\varepsilon\langle \breve \nabla_vA_i^\varepsilon, w\rangle
+\sum_i \breve \nabla_v A_i^\varepsilon\langle \breve \nabla_uA_i^\varepsilon, w\rangle.
\end{equation}
So by the conclusion of Proposition \ref{proposition4.1} and  the methodology  in Section 2, we obtain that
$$\lim_{\varepsilon\rightarrow 0}\E\sup_{0\leqslant s \leqslant t}|\hat W_s^{\varepsilon}-\hat W_s|^p=0 $$
for any $p>1$. %Furthermore from equation (\ref{l13}), we know $\hat W^{\varepsilon}$ has uniform lower bound in $\varepsilon$ it is straight forward to prove the $L^p$ convergence of $(\hat W_s^{\varepsilon})^{-1}$ by (\ref{l14}).
Since $W_t^{-1}$ also satisfies a linear SDE, but the driven Brownnian motion is with backward filtration, we derive
the $L^p$ convergence of $(\hat W_s^{\varepsilon})^{-1}$.
From that we know $\E\{ T\I^\varepsilon_t(h)| \sigma\{\xi^{\varepsilon}_s: 0\leqslant s\leqslant T\}\}$
converges to $\hat W_t \int_0^t (\hat W_s)^{-1}
A(\xi_s)\dot{h}_s ds$ in $L^p$ for any $p>1$. 

Note that for any $BC^1$ function $F$ on path space, we have
\begin{equation*}
\E\Big[T\I^\varepsilon_t(h) F(\xi^{\varepsilon}_{.}) \Big] =
\E\Big[\hat W_t^{\varepsilon}\left( \int_0^t (\hat W_s^{\varepsilon})^{-1}
A^{\varepsilon}(\xi^{\varepsilon}_s)\dot{h}_s ds\right)F(\xi^{\varepsilon}_{\cdot})\Big]
\end{equation*}
Let $\varepsilon$ tend to $0$, by the convergence results for $T\I^\varepsilon_t(h)$ %in Section \ref{bel-formula}  
we have,
\begin{equation*}
\E\Big[T\I_t(h) F(\xi_{.}) \Big] =
\E\Big[(\hat W_t \int_0^t (\hat W_s)^{-1}
A(\xi_s)\dot{h}_s ds)F(\xi_{.})\Big]
\end{equation*}
 which implies the conclusion (\ref{l12}).
\end{proof}

As the same approximaion argument, we can prove the following intrinsic integration by parts formula

\begin{thm}
We assume the  assumptions of Proposition \ref{proposition4.1}. Let $h:[0,1]\times\Omega\rightarrow \R^n$ be an adapted stochastic process with $h(\omega)\in L_0^{2,1}([0,1]; \R^n)$ for almost
surely all $\omega$ and $\E(\int_0^1
|\dot{h}_s|^2ds)^{\frac{1+\beta}{2}}<\infty$ for some $\beta>0$. Then
$$\E dF(\hat W_{.}h_{.})=\E \big[F(\xi_{.})\int_0^1 \langle \hat W_s \dot{h}_s, 
\breve \paral_s d \breve B_s\rangle\big]$$
for all $BC^1$ functions $F$ on path space.
%where $\bar V_t =\E\{ T\I_t(h)| \sigma\{\xi_s: 0\leqslant s \leqslant T\}\}$ and 
\end{thm}

\end{document}